\documentclass[11pt]{amsart}

\usepackage[ngerman,british]{babel} 
\usepackage[utf8]{inputenc}
\usepackage[T1]{fontenc}
\usepackage{amsthm}
\usepackage{amssymb}
\usepackage{mathtools} 
\usepackage{amsmath}
\usepackage[outline]{contour}
\usepackage[mathscr]{euscript}
\usepackage[all,2cell,cmtip]{xy}
\usepackage{csquotes}
\usepackage[usenames,dvipsnames]{xcolor}
\usepackage[colorlinks=true,allcolors=blue,linktocpage]{hyperref}
\usepackage{microtype}

\usepackage{tikz-cd}
\usepackage{amssymb}
\usetikzlibrary{calc}
\usetikzlibrary{decorations.pathmorphing}

\usetikzlibrary{babel} 

\tikzset{curve/.style={settings={#1},to path={(\tikztostart)
    .. controls ($(\tikztostart)!\pv{pos}!(\tikztotarget)!\pv{height}!270:(\tikztotarget)$)
    and ($(\tikztostart)!1-\pv{pos}!(\tikztotarget)!\pv{height}!270:(\tikztotarget)$)
    .. (\tikztotarget)\tikztonodes}},
    settings/.code={\tikzset{quiver/.cd,#1}
        \def\pv##1{\pgfkeysvalueof{/tikz/quiver/##1}}},
    quiver/.cd,pos/.initial=0.35,height/.initial=0}

\tikzset{tail reversed/.code={\pgfsetarrowsstart{tikzcd to}}}
\tikzset{2tail/.code={\pgfsetarrowsstart{Implies[reversed]}}}
\tikzset{2tail reversed/.code={\pgfsetarrowsstart{Implies}}}
\tikzset{no body/.style={/tikz/dash pattern=on 0 off 1mm}}

\usepackage[hyperref,
maxbibnames=99,
backend=biber,
style=alphabetic-verb,
citestyle=alphabetic-verb,
giveninits=true
]{biblatex}
\DefineBibliographyStrings{english}{
  page             = {},
  pages            = {},
} 
\renewbibmacro{in:}{} 
\usepackage{graphicx}
\usepackage{cleveref}   
\usepackage{underscore} 

\usepackage{import}
\usepackage{color} 
\usepackage{xifthen}

\usepackage{rotating}

\newcommand{\mcal}[1]{\mathcal{#1}}
\newcommand{\msc}[1]{\mathscr{#1}}

\newcommand{\mscc}[1]{\mathbf{#1}}
\newcommand{\mrm}[1]{\mathrm{#1}}
\newcommand{\mbf}[1]{\mathbf{#1}}
\newcommand{\mbb}[1]{\mathbb{#1}}
\newcommand{\mfrak}[1]{\mathfrak{#1}}
\newcommand{\down}[3][ ]{(#2 \downarrow #3)^{#1}}
\newcommand{\paren}[1]{\left( #1 \right)}

\newcommand{\R}{\mbb{R}}

\newcommand{\id}{{\mathrm{id}}}

\newcommand{\Hom}{\mathrm{Hom}}

\newcommand{\Vinj}{{\mscc{V}^{\mathrm{inj}}}}
\newcommand{\Vhook}{\mscc{V}^{\hookrightarrow}}

\newcommand{\pt}{{\mathrm{pt}}}

\newcommand{\Cinfty}{\mscc{C}\mathrm{at}_{\infty}}

\newcommand{\Map}{{\mathrm{Map}}}

\newcommand{\cdon}{\[\begin{tikzcd}}
\newcommand{\cdoff}{\end{tikzcd}\]}
\newcommand{\ncdon}{\begin{equation}\begin{tikzcd}}
\newcommand{\ncdoff}{\end{tikzcd}\end{equation}}
\newcommand{\OO}{{\mathrm{O}}}

\newcommand{\Sing}{{\mathrm{Sing}}}

\newcommand{\Nerve}{{\mathrm{N}}}

\newcommand{\Nhc}{\mathrm{N}^{\mathrm{hc}}}

\newcommand{\EEx}{\mathbf{EX}}

\newcommand{\Phat}{\mscc{P}^{\Delta}}

\newcommand{\Fun}{\mathrm{Fun}}

\newcommand{\Tangent}{\mathrm{T}}
\newcommand{\Normal}{\mathrm{N}}

\newcommand{\evv}{\mathrm{ev}}
\newcommand{\Ar}{\mathrm{Ar}}

\makeatletter
\newcommand{\btd}{\mathord{\mathpalette\raise@half\bigtriangledown}}
\newcommand\raise@half[2]{%
  \raisebox{.9\depth}{$\m@th#1#2$}%
}
\makeatother

\newcommand{\cat}[2]{\mscc{#1}\mathrm{#2}}

\newcommand{\ocat}[1]{\mathrm{#1}}
\newcommand{\Ctop}{\mscc{C}\mrm{at}_{\mrm{Top}}}

\newcommand{\tame}{\tau}

\numberwithin{equation}{section}

\theoremstyle{definition}
\newtheorem{definition}[equation]{Definition}
\newtheorem*{definition*}{Definition}

\newtheorem*{notation*}{Notation}
\newtheorem{construction}[equation]{Construction}
\newtheorem*{construction*}{Construction}

\theoremstyle{remark}
\newtheorem{remark}[equation]{Remark}
\newtheorem*{remark*}{Remark}
\newtheorem{example}[equation]{Example}
\newtheorem*{example*}{Example}

\newtheorem*{warning*}{Warning}

\newtheorem*{convention*}{Convention}

\theoremstyle{plain}
\newtheorem{proposition}[equation]{Proposition}
\newtheorem*{proposition*}{Proposition}
\newtheorem{theorem}[equation]{Theorem}
\newtheorem*{theorem*}{Theorem}
\newtheorem{definition-theorem}[equation]{Definition-Theorem}
\newtheorem*{definition-theorem*}{Definition-Theorem}

\newtheorem*{postulate*}{Postulate}
\newtheorem{lemma}[equation]{Lemma}
\newtheorem*{lemma*}{Lemma}

\newtheorem*{observation*}{Observation}
\newtheorem{corollary}[equation]{Corollary}
\newtheorem*{corollary*}{Corollary}

\newtheorem*{question*}{Question}

\title{{Linked spaces and exit paths}}

\author{\"Od\"ul Tet\.{i}k}
\thanks{This research was partially supported by the NCCR SwissMAP, funded by the Swiss National Science Foundation (SNF) and by the Austrian Science Fund (FWF) through SNF Grants no.\ 200020\_192080 and 200021\_227719 and FWF Project no.\ P 37046.}
\address{Institut für Mathematik, Universität Zürich, Winterthurerstrasse 190, 8057 Zürich, Switzerland}
\curraddr{University of Vienna, Faculty of Physics, Mathematical Physics Group, Boltzmanngasse 5, 1090 Vienna, Austria}
\email{oeduel.tetik@univie.ac.at}

\date{}

\begin{document}

\begin{abstract}
    Conically smooth spaces (CSSs), introduced by Ayala, Francis and Tanaka, constitute a large class of singular spaces including Whitney-stratified spaces. We reduce the stratified topology of CSSs over depth-$1$ posets to the ordinary topology of linked smooth manifolds, i.e., spans $M\xleftarrow{\pi} L\xrightarrow{\iota}N$ of smooth manifolds where $\pi$ is a fibre bundle and $\iota$ is a closed embedding. To that end, we introduce explicit exit path quasi-categories (EPCs) for linked spaces and prove that this induces a fully faithful functor from a quasi-category of linked spaces to the quasi-category of all quasi-categories whose essential image includes the Lurie--MacPherson EPCs of CSSs over depth-$1$ posets. We use linked smooth manifolds to resolve various weaker versions of a conjecture of Ayala--Francis--Rozenblyum in the negative by exhibiting quasi-categories with conservative functors to $\{0<1\}$ satisfying certain finiteness conditions but which are not equivalent to EPCs of CSSs. In a sequel, we develop a tangential theory for linked smooth manifolds and reduce the classification conically smooth bundles over depth-$1$ posets to that of ordinary bundles on linked smooth manifolds.
\end{abstract}

\maketitle

\tableofcontents

\section{Introduction}

Many singular spaces naturally decompose into collections of non-singular strata which are glued together in various ways. A main goal of stratified topology is to understand the stratified topology of singular spaces in terms of the ordinary topology of the strata and the gluing information. Typically, the latter is again parametrised by spaces, called (homotopy) links, which are not explicit in the original singular space but can be extracted from it. This process of decomposition is well-understood for many types of spaces.

The present paper is concerned with the inverse problem of explicitly reconstructing stratified homotopy types given systems of non-stratified spaces that are interpreted as strata and links. We achieve this in depth $1$ in a way that recovers and extends a large class of singular spaces.\footnote{By \emph{depth} we refer to the depth of the stratifying poset. The depths in the sense of \cite{ayala2017local} of the spaces we consider are otherwise arbitrary.} It is natural to expect that the methods we introduce extend to higher depth, but we leave this to future work.

Let $\mfrak{S}=(M\xleftarrow{\pi}L\xrightarrow{\iota}N)$ be a span of spaces. We call $\mfrak{S}$ a \emph{linked space} if $\pi$ is a fibration and $\iota$ is a cofibration. If we disregard the condition on $\iota$, such a span would arise immediately from any nice depth-$1$ filtered space $X=(M\subset \overline{N})$ by setting $N=\overline{N}\smallsetminus M$ and taking $L$ to be the space of paths in $X$ that start in $M$ and end in $N$, and which exit into $N$ immediately after time $0$. Such paths are called \emph{exit paths}.\footnote{See Quinn \cite{quinn1988homotopically}. The essential idea goes back at least to Fadell \cite{fadell1965normal}.} Similarly, paths $\gamma\colon[0,1]\to X$ such that $\gamma(t)\in N$ for $t<1$ and $\gamma(1)\in M$ are called \emph{enter paths}. An example of central importance to geometric topology and proper homotopy theory is the end space $e(Y)$ of an open manifold $Y$, the space of all proper maps $[0,\infty)\to Y$. It is equivalent to the space of all paths $\gamma\colon [0,\infty]\to Y^c=X$ to the one-point compactification of $Y$ such that $\gamma(\infty)=\infty$. If $Y$ is the interior of a compact manifold with boundary $\partial$, then there is a homotopy equivalence $\partial\simeq e(Y)$.\footnote{See Hughes--Ranicki \cite{hughesranicki1996ends} for a textbook account; also Porter \cite{porter1995proper}.} At the other extreme, the most general type of span we will consider in this paper is a \emph{linked $\infty$-category}, where $M$, $L$ and $N$ are $\infty$-categories (by which we mean quasi-categories) and $\pi$ is a right fibration.\footnote{Though in practice we will assume that $\iota$ is a cofibration, this assumption can be lifted in a straightforward manner as explained in \Cref{3F1JHL6}.} 

If $X$ is a conically smooth stratified space (CSS) in the sense of Ayala, Francis and Tanaka \cite{ayala2017local}, then $M$ and $N$ are smooth manifolds, and this path space is equivalent to an embedded normally-framed submanifold $\iota\colon L\hookrightarrow N$ of codimension $1$ which is simultaneously a fibre bundle $\pi \colon L\twoheadrightarrow M$, producing a linked smooth manifold. This is a very large class of stratified spaces which, by a result of Nocera and Volpe \cite{nocera2021whitney}, includes Whitney-stratified spaces \cite{whitney1965tangents} such as analytic varieties (which may have higher depth). We give two prototypical examples: If $X=(\partial M\subset M)$ is a smooth manifold with boundary, then $L\simeq\partial M$, reconstructing the end space of the interior. If $X=(M\subset \overline{N})$ as above is such that $\overline{N}$ is a smooth manifold and $M$ is a closed smooth submanifold of positive codimension, then $L=\mbb{S}\Normal_M$ can be taken to be the sphere bundle of the normal bundle of $M$, and $\iota$ to be induced by the choice of a tubular neighbourhood. In essence, a CSS $X$ of depth $n$ is a singular manifold with an atlas consisting of opens of type $C(Z)\times \mbb{R}^k$ where $C(Z)$ is the open cone of a compact CSS $Z$ of depth $n-1$, with smooth transition maps in an appropriate sense. A CSS of depth $0$ is a (paracompact and Hausdorff) smooth manifold.

The treatment of such stratified spaces and their stratified homotopy types in even greater generality started with unpublished work of MacPherson inspired by Goresky--MacPherson intersection homology and was developed by Treumann \cite{treumann2009exit}, Woolf \cite{woolf2009fundamental}, Miller \cite{miller2009popaths,miller2013strongly}, and Lurie \cite{luriehigheralgebra}. Subsequently, Ayala, Francis, Rozenblyum and Tanaka \cite{ayala2017local,ayala2016factorization,ayala2017stratified,ayala2020factorization} have developed Lurie's treatment of factorisation homology, a topological version of Beilinson--Drinfeld's chiral homology \cite{beilinson2004chiral}, in the conically smooth setting and with arbitrary tangential structure. In practice, it is vital to have a precise understanding of the stratified topology of a CSS $X$, and one would hope to apply non-stratified methods to it. Ideally, upon passing from $X$ to the span $\mfrak{S}$, the ordinary topology of $\mfrak{S}$ should give information about the stratified topology of $X$. The main purpose of the present paper is to provide an explicit inverse way to pass from $\mfrak{S}$ to $X$ at the level of stratified homotopy types. We work therefore mostly within topology in this paper. For a tangential theory for linked smooth manifolds and further applications to CSSs, see the sequel \cite{tetik2022stratified}.

\subsection{Main results}

We will start with a short overview of the related results in the literature.

Just as a locally constant sheaf on a space is tantamount to a functor out of the complex of its singular chains, a constructible (stratum-wise locally constant) sheaf on a conically stratified space is tantamount, by a result of Lurie \cite[\S A]{luriehigheralgebra}, to a functor out of its exit path $\infty$-category (EPC). Accordingly, the EPC should be thought of as the stratified homotopy type. In the EPC, only those paths which stay within a single stratum or which are exit paths are considered. The latter are by construction non-invertible. Ayala, Francis and Rozenblyum \cite{ayala2017stratified} showed that the EPC construction\footnote{AFR give a different but equivalent EPC construction, which is also equivalent to the opposite of yet another enter path $\infty$-category construction from \cite{ayala2017local}. Another EPC-like construction in the literature is Tamaki's topological face categories from \cite[\S 5]{tamaki2016cellularstratifiedspaces}.} defines a fully faithful functor 
\[
    \cat{E}{xit}\colon \cat{S}{trat}^{\mrm{cs}}\hookrightarrow\Cinfty
\]
from the $\infty$-category of conically smooth stratified spaces and conically smooth maps into the $\infty$-category of all $\infty$-categories.\footnote{See Lurie \cite[\S 3]{lurie2009higher} for the latter. We will recall it in \Cref{8592U1F}.} 

This raises the following question: How can one understand the EPC of $X$ in terms of the corresponding span $\mfrak{S}$? For $\mfrak{P}$ a poset, Douteau \cite{douteau2021homotopy} gave a Quillen equivalence between a certain model category of $\mfrak{P}$-stratified spaces and a model category of diagrams of simplicial sets indexed over non-degenerate sequences in $\mfrak{P}$. This bypasses EPCs.\footnote{See also Douteau--Waas \cite[Recollection 2.53 ff.]{douteau2021homotopylinks}.} 
In a similar vein, Barwick, Glasman and Haine \cite[\S 2]{barwick2020exodromy} gave, based on a well-known result of Joyal and Tierney \cite{joyaltierney2006quasisegal}, an equivalence
\[
    \mrm{N}_{\mfrak{P}}\colon\cat{S}{tr}_\mfrak{P}\xrightarrow{\sim}\cat{D}{\text{\'e}c}_{\mfrak{P}},
\]
a nerve-type construction,
from \emph{abstract $\mfrak{P}$-stratified homotopy types} to \emph{spatial d\'ecollages} over $\mfrak{P}$. Here, $\cat{S}{tr}_\mfrak{P}\subset(\Cinfty)_{/\Nerve(\mfrak{P})}$ is the full sub-$\infty$-category of the over-$\infty$-category of $\Cinfty$ over the nerve of the poset $\mfrak{P}$, generated by \emph{conservative} structure morphisms 
\[
    s\colon \cat{C}\to\Nerve(\mfrak{P})
\] 
which, by definition, have the property that the \emph{strata} $s^{-1}\{p\}$, $p\in\mfrak{P}$, are $\infty$-groupoids.
As for the target, when $\mfrak{P}=\{0<1\}$, a spatial d\'ecollage is a span $(M\leftarrow L\to N)$ of spaces with no conditions on the maps. In higher depth, spatial d\'ecollages are special versions of spatial diagrams in the sense of Douteau. In \cite[Theorem 2.2.20]{waas2025presentingtopologicalstratifiedhomotopy}, Waas presented this equivalence in terms of a Quillen equivalence involving generalised homotopy links.

Haine \cite{haine2023homotopytheorystratifiedspaces}  proved that there is an equivalence 
\[
    \cat{T}{op}_{/\mfrak{P}}[W^{-1}]\xrightarrow{\sim}\cat{S}{tr}_\mfrak{P}
\]
from a localisation of the category of $\mfrak{P}$-stratified topological spaces by passing though d\'ecollages instead of Lurie's EPC construction. In \cite[Theorem B, Corollary 4.4.5]{waas2025presentingtopologicalstratifiedhomotopy}, Waas proved a global version of this in terms of a Quillen equivalence using Lurie's EPC construction, implying an equivalence
\[
    \cat{E}{xit}\colon \cat{S}{trat}^{o}[H_s^{-1}] \xrightarrow{\sim} \cat{L}{ay}_{\infty}\subset\Cinfty
\]
from the category $\cat{S}{trat}^o$ of bifibrant stratified spaces including CSSs, with $H_s$ the class of stratified homotopy equivalences, to those $\infty$-categories in which every endomorphism is an isomorphism, called \emph{layered} $\infty$-categories

For our purposes and for the applications in the sequel \cite{tetik2022stratified}, the problem is that the functor $\mrm{N}_{\mfrak{P}}$ is an equivalence by virtue of being fully faithful and essentially surjective, so an inverse is hard to make explicit; e.g., Waas' Quillen inverse is a coend. Moreover, linked spaces and linked smooth manifolds have much more restrictive definitions than d\'ecollages. 

Therefore, we construct a fully combinatorial exit path $\infty$-category $\EEx(\mfrak{S})$ for every linked space $\mfrak{S}$ and prove that this defines a fully faithful functor from an $\infty$-category $\mscc{LS}$ of linked spaces to $\Cinfty$. In fact, $\EEx$ applies more generally to linked $\infty$-categories (see \Cref{thm:exit path category of linked space}). Our methods are necessarily independent of those of the works mentioned above.

\begin{theorem*}[{\ref{JG5DRK5}}]
    The EPC construction induces a fully faithful functor 
    \[
        \EEx\colon\mscc{LS}\hookrightarrow\Cinfty
    \]
    of $\infty$-categories.
\end{theorem*}

Our construction enables a direct and self-contained proof of fully-faithfulness.\footnote{A similar result can be obtained from \Cref{T47D64N} combined with \cite[Theorem 2.2.20 and Proposition 2.2.11]{waas2025combinatorialmodelsstratifiedhomotopy} through the language of model categories. This is neither used nor needed in the present paper.} We further illustrate the precise control which $\EEx$ yields by computing some fundamental two-point homotopy categories in \Cref{V3YDWXA,8PVG15K} in an elementary fashion. \Cref{5QWZ624} classifies local systems on linked spaces and, by \Cref{YSAL1EK} below, on CSSs of depth $1$.

The $\infty$-category $\mscc{LS}$ is the homotopy coherent nerve of the topological category $\mrm{LS}$ where 
\[
    \Hom_{\mrm{LS}}(\mfrak{S},\mfrak{S}')=\mcal{F}_0\amalg\mcal{F}_{1}\amalg\mcal{F}_{01},
\] 
for linked spaces $\mfrak{S}=(M\leftarrow L\to N)$ and $\mfrak{S}'=(M'\leftarrow L'\to N')$, is simply the disjoint union of the ordinary pullback spaces 
\begin{align*}
    &\mcal{F}_0=[M,M']\times_{[L,M']}[L,M']\times_{[L,M']}[N,M'],\\
    &\mcal{F}_1=[M,N']\times_{[L,N']}[L,N']\times_{[L,N']}[N,N'],\\
    &\mcal{F}_{01}=[M,M']\times_{[L,M']}[L,L']\times_{[L,N']}[N,N']
\end{align*}
where $[-,-]=\Hom_{\mrm{Top}}(-,-)$ takes the space of continuous maps with the compact-open topology. These model, respectively, the spaces of maps that map only to the lower stratum $M'$, only to the higher stratum $N'$, and to both strata. A systematic construction of $\mscc{LS}$ which includes more general linked spaces is given in \Cref{PCO1GNW}. For simplicity, in this section we will restrict ourselves to linked spaces of type $(M\leftarrow L\to N)$ and to CSSs stratified over $\{0<1\}$.

Before we explain the construction of $\EEx$, we will note a `linked smooth homotopy hypothesis' as an immediate corollary. Let $\mrm{LMfld}\subset\mrm{LS}$ denote the subcategory consisting of linked smooth manifolds where $\pi$ is a fibre bundle and $\iota$ is a closed embedding, where $\Hom_{\mrm{LMfld}}(\mfrak{S},\mfrak{S}')=\mcal{F}^{\infty}_{0}\amalg\mcal{F}^{\infty}_{1}\amalg\mcal{F}^{\infty}_{01}\subset\Hom_{\mrm{LS}}(\mfrak{S},\mfrak{S}')$ is defined as above except using $C^\infty$ maps. Observe that if $L\neq\emptyset$ then the maps $f\colon\mfrak{S}\to\mfrak{S}'$ are fully determined by their coordinate functions $f_N$ on the higher stratum, and conversely, every smooth map $f_N\colon N\to N'$ such that $f_N\circ\iota\colon L\to N'$ factors through $\iota'\colon L'\hookrightarrow N'$ determines a smooth map $f\colon\mfrak{S}\to\mfrak{S}'$ of linked smooth manifolds.\footnote{This is also true for continuous maps of linked spaces whose link projections are fibre bundles.} Thus $\mcal{F}^\infty_{0}\cong C^\infty(N,M')$, $\mcal{F}^\infty_{1}\cong C^\infty(N,N')$, and $\mcal{F}^\infty_{01}\cong C^\infty((N,L),(N',L'))\subset C^\infty(N,N')$ consists of maps rel $L$, $L'$ in the sense above. These spaces are weakly equivalent to their $C^0$ counterparts,\footnote{In order to apply Whitney approximation to $f\in\mcal{F}_{01}$, first apply it to $f_L$ to obtain a homotopic smooth replacement $f_L'$. A homotopy extension to $N$ yields a new map $f_N'$. Now apply Whitney approximation to $f_N'$ and use the fact that the homotopy can be chosen to be constant on the closed subspace $\iota(L)$, yielding a smooth map rel $L$, $L'$.} implying that $\mscc{L}\cat{M}{fld}\subset\mscc{LS}$ is in fact a \emph{full} sub-$\infty$-category.

\begin{corollary}
    The EPC construction restricts to a fully faithful functor 
    \[
        \EEx\colon\mscc{L}\cat{M}{fld}\hookrightarrow\Cinfty
    \]
    of $\infty$-categories from the $\infty$-category of linked smooth manifolds and smooth maps between them.
\end{corollary}

The basic idea of $\EEx$ is that a path $\gamma$ in $N$ that starts in $\iota(L)$ may be adjoined formally as a non-invertible $1$-morphism to the Kan complex $\mathrm{Sing}(M)\amalg\mrm{Sing}(N)$ as one that starts at $\pi(\gamma(0))$ and ends at $\gamma(1)$. Higher-dimensional simplices that connect $M$ and $N$ may be adjoined in a similar fashion, but one must keep track of the \emph{exit index} of an exit path. We do this by means of $(1,n-1)$-shuffles. Such a shuffle is determined uniquely by a number in $\{1,\dots,n\}$, the exit index, which corresponds to the time at which an $n$-path exits into $N$. 
The fact that we only consider shuffles of this type corresponds to our restriction to depth $1$. In depth $k$, one must use $(k,n-1
)$-shuffles, whose `exit index' is a subposet $[k-1]\hookrightarrow\{1<\dots<n\}$. This is an incarnation of the idea of generalised links as treated by Douteau and Henriques (see \cite{douteau2021homotopy,henriquesstratified}) so that one may interpret $\EEx$ as a way of (combinatorially) specifying a point-set EPC directly from such data. On the other hand, one could use the depth-$1$ construction iteratively since it applies to linked $\infty$-categories. We expect these two methods to produce equivalent stratified homotopy types.

The EPC of the span $\mfrak{S}_X$ naturally associated with a depth-$1$ CSS $X$ is equivalent to $\cat{E}{xit}(X)$, the EPC constructed by Ayala--Francis--Rozenblyum \cite{ayala2017stratified}. We show this by means of an explicit comparison map $\EEx(\mfrak{S}_X)\to\cat{E}{xit}(X)$ of quasi-categories to the Lurie--MacPherson model from \cite{luriehigheralgebra}. The latter is equivalent to that of AFR by \cite[Lemma 3.3.9]{ayala2017stratified}. 

\begin{theorem*}[{\ref{YSAL1EK}}]
    If $X$ is a CSS of depth $1$, then $\EEx(\mfrak{S}_X)\simeq \cat{E}{xit}(X)$.
\end{theorem*}

The comparison map uses the unzip construction of \cite{ayala2017local}; more generally, \Cref{YSAL1EK} holds for the Lurie--MacPherson EPCs of all depth-$1$ spaces `of unzip type' as in \eqref{A8U7AOO} which need not be (conically) smooth.

Since $\EEx$ applies to linked $\infty$-categories, it can be iterated to model the EPCs of higher-depth spaces without requiring the natively higher-depth treatment explained above.

A main result of \cite{ayala2017stratified} is that $\cat{E}{xit}$ is a fully faithful functor from conically smooth spaces into $\infty$-categories, where $\Hom(X,X')$ for two such spaces is the space of conically smooth maps $X\to X'$ in the sense of Ayala--Francis--Tanaka \cite{ayala2017local}.\footnote{An explicit description of $\Hom(X,X')$ as a Kan complex is given in \cite[Corollary 2.4.3]{ayala2017stratified}.} These results together imply an improvement on the depth-$1$ conically smooth approximation conjecture \cite[Conjecture 1.5.1]{ayala2017local} which states that there is an equivalence $\Hom_{\cat{S}{trat}^{\mrm{cs}}}(X,X')\simeq\Hom_{\cat{S}{trat}}(X,X')$, where $\cat{S}{trat}$ is the $\infty$-category of $C^0$-stratified spaces. The following reduces conically smooth maps, up to weak equivalence, further to non-stratified topological maps.

\begin{corollary}
    Let $X$ and $X'$ be conically smooth stratified spaces over $[1]$, and let $\mfrak{S}=(M\leftarrow L\to N)$, $\mfrak{S}'=(M'\leftarrow L'\to N')$ be the associated linked smooth manifolds. Then the space of conically smooth maps $X\to X'$ satisfies weak equivalences
    \[
        \Hom_{\cat{S}{trat}^{\mrm{cs}}}(X,X')\simeq\Hom_{\mscc{L}\cat{M}{fld}}(\mfrak{S},\mfrak{S}')\simeq \Hom_{\mscc{LS}}(\mfrak{S},\mfrak{S}')\simeq \Hom_{\cat{S}{trat}}(X,X').
    \]
\end{corollary}

We expect that an analogous result holds for conically smooth spaces of arbitrary depth. The first equivalence uses the fully-faithfulness result of AFR, and the last uses the results of Haine--Waas discussed above. The original form of the approximation conjecture, the outer equivalence, also follows from results of Waas \cite{waas2025presentingtopologicalstratifiedhomotopy} given AFR's theorem.

We will now list the basic properties of $\EEx$, suppressing the cofibration. Due to the explicit definition of $\EEx$, the constructions involved in the proofs of the following statemens are likewise explicit. Given two maps $A,B\to C$ of simplicial sets, we write $\down{A}{B}=\down[C]{A}{B}=A\times_{C^{\Delta\{0\}}}C^{\Delta[1]}\times_{C^{\Delta\{1\}}}B$ for the space of all morphisms in $C$ that start in the image of $A$ and end in the image of $B$.\footnote{This is in contrast to the $\infty$-categorical homotopy fibre product which will be empty in cases of interest.} It is the homotopy link between $A$ and $B$.

\begin{theorem*}[{\ref{EY7G1PB}, \ref{T47D64N}, \ref{J74MRSH}, \ref{RPISW40}}]
    Let $\mfrak{S}=({M}\twoheadleftarrow{L}\hookrightarrow{N})$ be a linked $\infty$-category.
    \begin{enumerate}
        \item The objects of $\EEx(\mfrak{S})$ are those of ${M}$ and ${N}$.
        \item The $\infty$-categories ${M}$ and ${N}$ are full sub-$\infty$-categories of $\EEx(\mfrak{S})$.
        \item There are no morphisms in $\EEx(\mfrak{S})$ from ${N}$ to ${M}$.
        \item For $p\in{M}$ and $q\in{N}$, the morphism space $\Hom_{\EEx(\mfrak{S})}(p,q)\simeq \mscc{P}_{{L}_p,q}$ is equivalent to the space of morphisms in ${N}$ that start in the embedded fibre ${L}_p$ and end in $q$.
        \item There is an isomorphism $\down[\EEx]{{M}}{{N}}\cong\down[N]{{L}}{{N}}$.
        \item There is a constant exit loop inclusion $\Box\colon {L}\hookrightarrow\down{{M}}{{N}}$.
        \item If $\mfrak{S}$ is a linked $\infty$-groupoid, then $\Box$ is an equivalence: $\down{{M}}{{N}}\simeq {L}$.
    \end{enumerate}
\end{theorem*}

These properties are analogous to facts already known in different versions of stratified space theory. In the context of homotopically stratified spaces in the sense of Quinn \cite{quinn1988homotopically}, Miller showed in \cite[Theorem 6.3]{miller2013strongly} that stratified homotopy equivalences between such spaces are exactly those maps which induce weak equivalences on strata and homotopy links. An analogous result for abstract poset-stratified homotopy types was obtained by Barwick--Glasman--Haine \cite[Theorem 2.7.4]{barwick2020exodromy}, as discussed above. In the conically smooth setting, AFR idenitifed path spaces between strata in terms of links in \cite[Lemma 3.3.5]{ayala2017stratified}. That equivalences of EPCs of CSSs are checked on strata and (generalised) links is also implied by Haine's main theorem in \cite{haine2023homotopytheorystratifiedspaces}.

\subsection{Counterexamples}

Using \Cref{YSAL1EK,JG5DRK5}, we obtain three classes of counterexamples to versions of a conjecture of Ayala--Francis--Rozenblyum \cite[Conjecture 0.0.8]{ayala2017stratified} which characterises the essential image of $\cat{E}{xit}$ on CSSs as those finite $\infty$-categories in which every endomorphism is an isomorphism.\footnote{A (homotopically) finite space is one that has the homotopy type of a finite CW complex. See Volpe \cite{volpe2024finitenessfinitedominationstratified} for a salient treatment in the $\infty$-categorical setting and concerning EPCs.} An $\infty$-category with a conservative functor to (the nerve of) a poset $P$ is called \emph{$P$-layered}. The finiteness of the $\infty$-categories constructed below follows immediately from a result of Volpe \cite[Proposition 2.11]{volpe2024finitenessfinitedominationstratified}. The \emph{(homotopy) link} of a $[1]$-layered $\infty$-category $\mscc{C}$ is the space $\down{\mscc{C}_0}{\mscc{C}_1}$ as above, and its \emph{local links} are the homotopy fibres of the evaluation map $\down{\mscc{C}_0}{\mscc{C}_1}\to\mscc{C}_0$. An example similar to the first one was already given by Volpe \cite[Remark 2.14]{volpe2024finitenessfinitedominationstratified}.

\begin{corollary*}[{\ref{JWCBC9T}}]
    Let $B$ be a smooth manifold without boundary whose fundamental group is not purely torsion, let $\widetilde{B}\to B$ be its universal cover, and let $\widetilde{B}\to Y$ be a cofibration. Then the $\infty$-category $\EEx(B\leftarrow \widetilde{B}\to Y)$ is not equivalent to the EPC of a CSS. If $B, \widetilde{B}$ and $Y$ are finite, then so is $\EEx(B\leftarrow \widetilde{B}\to Y)$.
\end{corollary*}

\begin{corollary*}[{\ref{DM62IBZ}}]
    Let $\Lambda$ be compact smooth manifold with a non-vanishing Stiefel--Whitney class, and let $\Lambda\hookrightarrow \mbb{R}^K$ be a closed embedding. Then the finite $[1]$-layered $\infty$-category $\EEx(\ast\leftarrow \Lambda\xhookrightarrow{} \mbb{R}^K)$ has contractible strata, finite local links, and is not equivalent to the EPC of a CSS.
\end{corollary*}

\begin{corollary*}[{\ref{761A5V7}}]
    Let $Y$ be a closed smooth manifold which is not a homology sphere and let $Y\hookrightarrow\mbb{R}^K$ be a closed embedding. Then the finite $[1]$-layered category $\EEx(Y\xleftarrow{\id}Y\hookrightarrow\mbb{R}^K)$ has contractible local links, contractible higher stratum, and is not equivalent to the EPC of a compact CSS.
\end{corollary*}

\begin{remark}\label{H09VB9M}
    AFR's conjecture is understood to mean that the essential image of the EPC functor on CSSs can be characterised by finding appropriate finiteness conditions. The examples listed above disprove, respectively, the following weaker versions. In each case, the homotopy links are also finite.
    \begin{enumerate}
        \item Every finite $[1]$-layered $\infty$-category is equivalent to the EPC of a CSSs.
        \item Every finite $[1]$-layered $\infty$-category with finite local links (and contractible strata) is equivalent to the EPC of a CSS.
        \item Every finite $[1]$-layered $\infty$-category with contractible local links (and contractible higher stratum) is equivalent to the EPC of a compact CSS.
    \end{enumerate}
    Such finiteness conditions are therefore insufficient to characterise the EPCs of CSSs.
\end{remark}

\subsection{Classification of conically smooth bundles}
In the sequel \cite{tetik2022stratified} we show the following for every pair $n,m\geq0$ of natural numbers: 
\begin{theorem*}[{\cite[Theorem 5.3]{tetik2022stratified}}]
    There exists a fully faithful functor 
    \[
        \EEx(B\OO(n,m))\hookrightarrow\Vhook
    \]
    of $\infty$-categories where the domain is the EPC of the linked space 
    \[
        B\OO(n,m)=(B\mathrm{O}(n)\twoheadleftarrow B\mathrm{O}(n)\times B\mathrm{O}(m)\xhookrightarrow{\oplus} B\mathrm{O}(n+m)).
    \]
\end{theorem*}
We call $B\OO(n,m)$ the (infinite) \emph{$(n,m)$-Grassmannian}. The target is a quasi-category model of the complete Segal space $\Vinj$, the stratified infinite Grassmannian of Ayala--Francis--Rozenblyum \cite{ayala2020factorization}. Conically smooth vector bundles on a CSS $X$ are classified by $\infty$-functors $\cat{E}{xit}(X)\to\Vinj$. Consequently, if $X$ is a CSS with associated linked smooth manifold $\mfrak{S}_X=(M\leftarrow L\to N)$, then using \Cref{YSAL1EK,JG5DRK5} we have the following classification of vector bundles on $X$ which does not depend on the fully-faithfulness of AFR's $\cat{E}{xit}$:
\begin{corollary*}[{\cite[Corollary 1.3]{tetik2022stratified}}] 
    Let $X$ be a depth-$1$ CSS and $\mfrak{S}=(M\leftarrow L\to N)$ the associated linked space. Then the moduli space
    \[
        \mscc{E}_{n,m}(X)\subset\Fun^{\simeq}(\cat{E}{xit}(X),\Vinj)=\Hom_{\Cinfty}(\cat{E}{xit}(X),\Vinj)
    \] 
    of conically smooth vector bundles $E$ on $X$ satisfying $\mrm{rk}({E}|_M)=n$ and $\mrm{rk}({E}|_N)=n+m$ is weakly equivalent to the space of topological span maps $\mfrak{S}\to B\OO(n,m)$:
    \[
        \mscc{E}_{n,m}(X)\simeq [\mfrak{S},B\OO(n,m)].
    \]
    Similarly, the full moduli space admits a weak equivalence 
    \[
        \mscc{E}(X)=\Hom_{\Cinfty}(\cat{E}{xit}(X),\Vinj)\simeq\coprod_{n,m\geq0}[\mfrak{S},B\OO(n,m)].
    \]
\end{corollary*}
Moreover, the tangent classifier $\Tangent X\colon\cat{E}{xit}(X)\to\Vinj$ of a depth-$1$ CSS $X$ satisfying $\dim(M)=n$ and $\dim(N)=n+m$ reduces to a particular map $\Tangent \mfrak{S}\colon\mfrak{S}\to B\OO(n,m)$ that is constructed explicitly in \cite{tetik2022stratified} following an idea of AFR (see \cite[Corollary 1.2]{tetik2022stratified}). 

As for a classification of conically smooth principal bundles, let $\{G(n)\}_{n\geq 0}$ be a collection of topological groups together together with an operation $\boxplus\colon G(n)\times G(m)\to G(n+m)$ which is a cofibration, and which induces an $\mbb{E}_\infty$-structure on $BG_\amalg=\coprod_{n\geq0}BG(n)$. There is a strictification to a topological monoid $(BG_\amalg^{\infty},\boxplus)$ explained in \cite[\S 4.1]{tetik2022stratified}, the idea behind which is similar to Schwede's treatment of symmetric monoid-valued orthogonal spaces in \cite[\S 2]{schwede2018global}. Analogously to $\Vhook$, we can construct the classifying $\infty$-category 
\(
    \mscc{G}^{\hookrightarrow}=\ast/\Nhc(B^{\boxplus}G)
\)
of conically smooth $G$-bundles (of arbitrary depth)
as the under-$\infty$-category of the single-object homotopy-coherent nerve of the delooping $B^{\boxplus}G$ of $(BG_{\amalg}^\infty,\boxplus)$. If there are maps $G(n)\to \OO(n)$ such that the induced operation $BG(n)\times BG(m)\to BG(n+m)$ covers $\oplus\colon B\OO(n)\times B\OO(m)\to B\OO(n+m)$ (i.e., if the system $\{G(n)\to\OO(n)\}$ is \emph{multiplicative}), then there is an induced map 
\(
    \mscc{G}^{\hookrightarrow}\to\Vhook
\)
of $\infty$-categories.
Classical structure groups provide examples.

We now obtain as a corollary that conically smooth principal (multiplicative) $G$-bundles on a depth-$1$ CSS $X$ are classified exactly by topological span maps $\mfrak{S}_X\to BG(n,m)=(BG(n)\leftarrow BG(n)\times BG(m)\to BG(n+m))$ and that the classification of $G$-structures on a conically smooth vector bundle $E$ reduces to the study of span-map lifts of $E\colon\mfrak{S}_X\to B\OO(n,m)$ along $BG(n,m)\to B\OO(n,m)$ for appropriate $n,m\geq0$. A precise statement is given in \cite[Corollary 1.5]{tetik2022stratified}. This opens the door to an explicit obstruction theory, which we leave to future work.

\subsubsection*{Acknowledgments} I have benefitted from exchanges with David Ayala, K.\ \.{I}lker\ Berktav, Nils Carqueville, Alberto S.\ Cattaneo, Owen Gwilliam, Aleksandar Ivanov, Thomas Leh\'ericy, Noam Szyfer, Hiro L.\ Tanaka, Marco Volpe, and Lukas Waas.

\subsubsection*{Conventions} The set $\mbb{N}$ of natural numbers includes zero. We denote the real line by $\R$. The standard simplicial $n$-simplex is $\Delta[n]\coloneqq\Hom_{\mathbf{\Delta}}(-,[n])$, where $\mathbf{\Delta}$ is the simplex category, and $[n]=\{0<\cdots<n\}$. The standard topological $n$-simplex is denoted by $\Delta^n$. For a sub-poset $S\subseteq[n]$, we sometimes abbreviate $\Delta[|S|]\hookrightarrow\Delta[n]$, specified by the inclusion, to $\Delta S$ or to $S$. By an $\infty$-category we mean a quasi-category, i.e., a simplicial set satisfing the weak Kan property. By an $\infty$-groupoid we mean a Kan complex. Spaces are compactly generated and weakly Hausdorff. A cofibration of simplicial sets is a monomorphism, and a right fibration ${C}\to{D}$ satisfies the right lifting property for horns $\Lambda^n_i\to{D}$ where $0<i\leq n$.

\section{Shuffles and exit indices}

Let ${M}\xleftarrow{\pi}{L}\xrightarrow{\iota}{N}$ be a span of simplicial sets. In this section, we will provide the combinatorial background that will enable us to adjoin non-invertible paths from $M$ to $N$ using $L$.

\begin{definition}\label{dfn:paths starting in a subspace}
    The simplicial set
    \(
        \mscc{P}\coloneqq\mscc{P}_{\iota}\coloneqq{L}\times_{{N}^{\{0\}}}{N}^{\Delta[1]}
    \)
    is called the \emph{mapping cocylinder} of $\iota$.
\end{definition}

\begin{remark}\label{4S91CIG}
    Recall how the mapping cocylinder appears in classical topology: in the analogous construction with spaces $L,N$ and $\iota$ a continuous map, the target evaluation $P_\iota\rightarrow N$, the path space fibration, is a fibration replacement for $\iota$ in view of a homotopy equivalence $L\simeq P_\iota$. 
\end{remark}

There are two induced maps $\pi,\tau\colon \mscc{P}\rightarrow{M},{N}$ defined as the compositions $\mscc{P}\to{L}\to {M}$ and $\mscc{P}\to{N}^{\Delta[1]}\to{N}^{\{1\}}$, respectively.

\begin{remark}\label{rmk:shuffles} A vertex of $\mscc{P}$ is a path of ${N}$ that starts at a point in $\iota({L})$. One may view this as a path which starts in ${M}$, by projecting to ${M}$ via $\pi$, and which, analogously, ends in ${N}$ via $\tau$. For higher morphisms, however, a direct generalisation of this idea requires unnatural choices. For instance, a $1$-morphism in $\mscc{P}$ may be depicted as 
\begin{equation}\label{eq:1-morphism mapping cocylinder}
    {\begin{tikzcd}
    \bullet \ar[r] & \bullet\\
    \bullet\ar[ur,dashed,color=blue] \ar[u] \ar[r] &\bullet \ar[u]
    \end{tikzcd}
    }
\end{equation}
where the bottom edge is in $\iota({L})$, and the top edge is in ${N}$. The two possibly non-degenerate $2$-simplices of ${N}$ we may extract the two triangles
corresponding to the two $(1,1)$-shuffles 
\(
    \Delta[2]\hookrightarrow\Delta[1]\times\Delta[1]
\)
\`a la Eilenberg--Mac Lane--Zilber \cite[]{eilenberg1953products,eilenberg1953groups} (see also Lurie \cite[00RF]{kerodon}, and below). If we were to add \eqref{eq:1-morphism mapping cocylinder} as a $2$-morphism to ${M}\amalg{N}$, say with source edge the bottom one, then we would have to choose the hypotenuse of the left triangle as the target edge, and the vertical edge as the intermediate $\overline{12}$-edge. But we may equally well make the analogous choice with  the right triangle, declaring the left vertical edge the source. However, both types of triangle are required for composition: if we wish to compose a path in ${M}$ with a (non-invertible) $1$-morphism in $\mscc{P}$, then we need (assuming there is a lift to ${L}$) a triangle of the first type. Similarly, composing a non-invertible $1$-morphism with a path in ${N}$, requires a triangle of the second type.
\end{remark}

\begin{definition}\label{constr:exit shuffles}
    Any pair $1\leq j\leq k$ of natural numbers determines a $(1,k-1)$-shuffle $\mcal{S}^k_{j}=\mcal{S}_j\colon\Delta[k]\hookrightarrow \Delta[1]\times\Delta[{k-1}]$ given by the non-degenerate element of $\left(\Delta[1]\times\Delta[{k-1}]\right)_k$ induced in $\mathbf{\Delta}$ by the poset map 
    \(
        [k]\rightarrow[1]\times[k-1]
    \)
    defined by 
    \[
        i\mapsto
        \begin{cases}
            (0,i), & i<j\\
            (1,i-1), & i\geq j.
        \end{cases}
    \]
    We call $\mcal{S}_j$ an \emph{exit shuffle}, and $j$ its \emph{exit index}. It has multiple left inverses, but we will use a particular one, $\mcal{C}^k_j=\mcal{C}_j$, induced in $\mathbf{\Delta}$ by the poset map 
    \(
        [1]\times[k-1]\twoheadrightarrow[k]
    \) 
    given by 
    \begin{align*}
        (0,i)\mapsto
        \begin{cases}
            i, &i<j\\
            j-1, &i\geq j
        \end{cases}
        ,
        \quad 
        (1,i)\mapsto
        \begin{cases}
            j, & i<j\\
            i+1, & i\geq j
        \end{cases}
        .
    \end{align*}
    This choice is justified below by \Cref{lem:meaning of exit index,lem:two classes} which otherwise fail.
\end{definition}

\begin{definition}
    \label{dfn:induced exit paths}
    Let $\iota\colon{L}\rightarrow{N}$ be a map of simplicial sets. 
    For $k\geq1$, we define 
    \[
    \Phat_{k-1}\subset {N}_{k}\times\{1,\dots,k\}
    \]
    to be the subset consisting of pairs $(\gamma,j)$ such that in the diagram
    \begin{equation*}
        \begin{tikzcd}
            \Delta[k]\ar[dr,bend right=10,"{\mcal{S}_j}"',hook]\ar[rr,"{\gamma}"] & & {N}\\
             & \Delta[{1}]\times\Delta[{k-1}]\ar[ul,bend right=10,"{\mcal{C}_j}"',two heads]\ar[ur,dashed,"{\Gamma=\gamma\circ\mcal{C}_j}"']
        \end{tikzcd}
    \end{equation*}
    the arrow $\Gamma$ is in the image of the natural map $\mscc{P}\rightarrow{N}^{\Delta[1]}$. A pair $(\gamma,j)\in\Phat_{k-1}$ is an \emph{exit $k$-path of index $j$}.
\end{definition}

    \Cref{constr:exit shuffles} corresponds to the following phenomenon: a stratified $k$-chain $\Delta^{k}\rightarrow X$ of $X$ is a map of poset-stratified spaces, where $\Delta^{k}=\overline{C}^k(\pt)$ is the $k$-fold closed cone on the point. The closed cone $\overline{C}(Y)$ (see also \Cref{2J0Q1NN}) of a stratified space $Y\rightarrow\mfrak{P}=\mfrak{P}_Y$, where $\mfrak{P}$ is the stratifying poset (equipped with the Alexandrov topology where downward-closed subsets are closed) has $\pt\coprod_{\{0\}\times Y}[0,1]\times Y$ as its underlying space, and $\mfrak{P}_{\overline{C}(Y)}=\mfrak{P}_{{Y}}^{\triangleleft},$ i.e., $\mfrak{P}_{Y}$ with a minimal element adjoined, as its stratifying poset, together with the obvious stratification $\overline{C}(Y)\rightarrow\mfrak{P}_{Y}^{\triangleleft}$. Now, the stratified map $\Delta^{k}\rightarrow X$ is a commutative topological square
    \begin{equation*}
        \begin{tikzcd}
            \Delta^{k} \ar[d] \ar[r,"{f}"] & X \ar[d]\\
            \mfrak{P}_{\Delta^k} \ar[r,"{s_f}"'] & \mfrak{P}_X
        \end{tikzcd}
        .
    \end{equation*}
    Clearly we have $\mfrak{P}_{\Delta^k}\simeq[k]$ as posets. If \(P_X\simeq\{a\prec b\},\) then the poset map $s_f$ is determined by a unique minimal `exit index' $j\in[k]$. Namely, let $j=0$ if $s_f$ is constant, or else let $j$ be the smallest number such that \(s_f(j-1\prec j)=a\prec b\).

\begin{definition}\label{dfn:vertical bottom top}
    Let $k\geq1$, $(\gamma,j)\in\Phat_{k-1}$.
    Then $d_i(\gamma)$ is \emph{low} if 
    \[
        \mcal{S}_j\circ\partial_i\colon\Delta[k-1]\hookrightarrow\Delta[k]\to\Delta[1]\times\Delta[k-1]
    \]
    factors through $\{0\}\times\Delta[k-1]$; \emph{upper} if it factors through $\{1\}\times\Delta[k-1]$; and \emph{vertical} if it is neither low nor upper.
\end{definition}

The terms `low' and `upper' are borrowed from \cite{eilenberg1950semi}. In the exit path $\infty$-category of \Cref{dfn:exit of span} below, vertical faces will remain non-invertible, low faces will become simplices in ${M}$, and upper faces in ${N}$. 
Writing `$d_i(\gamma)$ is vertical', etc., is slightly redundant as well as abusive, since these properties depend only on $i$ and $j$.

\begin{definition}\label{dfn:bemol exit index of face}
    Let $k\geq1$. We write
    \begin{align*}
        \text{\(
    \flat^k_{j,i}=\flat_{j,i}\in[k-1],
    \) 
    resp.\ \(\sharp^k_{j,i}=\sharp_{j,i}\in[k]\)}
    \end{align*} 
    for the smallest number whose image under $\mcal{S}_j\partial_i\colon[k-1]\rightarrow[1]\times[k-1]$, resp.\ under $\mcal{S}_j\sigma_i\colon[k+1]\rightarrow[1]\times[k-1]$ has first coordinate $1$. We leave $\flat^k_{k,k}$ undefined.
\end{definition}

See \eqref{eq:formula for flat} and \eqref{eq:formula for sharp} below for formulas. The following is a direct check: 

\begin{lemma}\label{lem:meaning of exit index}
Let $k\geq2$ and suppose $(j,i)\neq(k,k)$. The maps \[\mcal{S}_j\circ\partial_i\circ\mcal{C}_{\flat^k_{j,i}}\colon \Delta[1]\times\Delta[k-2]\to\Delta[1]\times\Delta[k-1]\] and \[\mcal{S}_j\circ\sigma_i\circ\mcal{C}_{\sharp^{k}_{j,i}}\colon\Delta[1]\times\Delta[k]\to\Delta[1]\times\Delta[k-1]\] preserve the first coordinate.
\end{lemma}

\begin{lemma}\label{lem:two classes}
    Let $k\geq2$. If $(\gamma,j)\in\Phat_{k-1}$ and $d_i(\gamma)$ is vertical, then 
    \(
        d_i(\gamma,j)\coloneqq (d_i\gamma,\flat^{k}_{j,i})\in\Phat_{k-2}.
    \)
\end{lemma}
\begin{proof}
    Consider the following commutative diagram:
    \begin{equation}\label{eq:diagram of a face}
        \begin{tikzcd}
            \Delta[{k-1}] \ar[ddrr,hook,"{\mcal{S}_\flat}"', bend right=10] \ar[r,hook,"{\partial_i}"] & \Delta[{k}] \ar[dr,hook,"{\mcal{S}_j}"', bend right=10] \ar[rr,"{\gamma}"] & & {N}\\
            && \Delta[{1}]\times\Delta[{k-1}] \ar[ul,two heads,"{\mcal{C}_j}"',bend right=10] \ar[ur,dashed,"{\Gamma}"] & \\
            && \Delta[1]\times\Delta[{k-2}]\ar[uull,two heads,"{\mcal{C}_\flat}",bend right=10] \ar[u,dashed,"{d'}"'] \ar[uur,dashed,"{\Gamma'}"',bend right] & 
        \end{tikzcd}
    \end{equation}
    \Cref{lem:meaning of exit index} implies in particular that the restriction of $d'=\mcal{S}_i\partial_i\mcal{C}_{\flat}$ to $\{0\}\times\Delta[{k-2}]$ factors through $\{0\}\times\Delta[{k-1}]$, which implies that $\Gamma'=\Gamma d'$ lifts to $\mscc{P}_{k-2}$, as desired. The case $j=i=k$ is precluded by verticality.
\end{proof}

The following variant for degeneracies is analogous. 

\begin{lemma}\label{lem:degenerate exit} Let $k\geq1$. 
    If $(\gamma,j)\in\Phat_{k-1}$, then
    \(
        s_i(\gamma,j)\coloneqq(s_i\gamma,\sharp^k_{j,i})\in\Phat_{k}.
    \)
\end{lemma}

\section{Exit paths of linked quasi-categories}\label{sec:exit paths}

    We are now ready to give the main construction of this paper. 
    First, note that if $\iota\colon{L}\hookrightarrow{N}$ is a cofibration, then an exit path $(\gamma,j)\in\Phat_{k-1}$ determines a canonical $(k-1)$-simplex $\Delta[{k-1}]\rightarrow{L}$ of ${L}$, that is, the restriction of $\Gamma=\gamma\circ\mcal{C}_j$ along $\{0\}\times\Delta[{k-1}]\hookrightarrow\Delta[{1}]\times\Delta[{k-1}]$ factors then through ${L}$ \emph{uniquely}. \Cref{MPD0BDL} will lift this assumption.

\begin{definition}\label{dfn:exit of span}
    Let \(\mathfrak{S}=({M}\xleftarrow{\pi}{L}\xhookrightarrow{\iota}{N})\) be a span of simplicial sets be given, where $\iota$ is a cofibration. We define a new simplicial set $\EEx=\EEx(\mathfrak{S})$ as follows:
    \begin{itemize}
        \item $\EEx_0={M}_0\amalg{N}_0$.
        \item
        \(
            \EEx_k={M}_k\amalg \Phat_{k-1} \amalg{N}_k
        \)
        for $k\geq1$.
        \item Face and degeneracy maps restricted to ${M}_k$ and ${N}_k$ are those of ${M}$ and ${N}$.
        \item For $k=1$ and $\gamma=(\gamma,1)\in\Phat_0\subset{N}_1$, we set
        \begin{itemize}
            \item $d_1(\gamma,1)=\pi (d_1\gamma)\in{M}_0$, 
            \item $d_0(\gamma,1)= d_0\gamma\in{N}_0$.
        \end{itemize}
        \item For $k\geq2$, $(\gamma,j)\in\Phat_{k-1}$, and $d_i$ a face map: 
        \begin{itemize}
            \item if $d_i\gamma$ is vertical, then
                \(
                    d_i(\gamma,j)=(d_i\gamma,\flat_{j,i})\in\Phat_{k-2}.
                \)
            \item if $d_i\gamma$ is low, then $d_i(\gamma,j)=\pi (d_i\gamma)\in{M}_{k-1}$.
            \item if $d_i\gamma$ is upper, then  $d_i(\gamma,j)= d_i\gamma\in{N}_{k-1}$.
        \end{itemize}
        \item If $k\geq1$, $(\gamma,j)\in\Phat_{k-1}$, and $s_i$ is a degeneracy, then \(
            s_i(\gamma,j)=(s_i\gamma,\sharp_{j,i})\in\Phat_k
        \).
    \end{itemize}
\end{definition}

An equivalent description of $\EEx$ without using shuffles is given in \Cref{0K44G88}.

\begin{lemma}\label{Q9W21WG}
    $\EEx(\mathfrak{S})$ is a simplicial set.
\end{lemma}

\begin{proof}
    We will verify the simplicial identities. Let $(\gamma,e)\in\Phat_{k-1}$. We will assume $k\geq2$ or $k\geq3$ depending on the applicability of the identity in question, and leave the case $k=1$ to the reader.

    \noindent\underline{$d_id_j=d_{j-1}d_i$ for $i<j$:}  Using 
    \begin{equation}\label{eq:formula for flat}
        \flat^k_{e,j}=
        \begin{cases}
            e, & j\geq e\\
            e-1, & j<e
        \end{cases}
    \end{equation}
    it is straightforward to see that
    \begin{equation}\label{eq:flat identity}
        \flat^{k-1}_{\flat^k_{e,j},i}=\flat^{k-1}_{\flat^k_{e,i},j-1}
    \end{equation}
    when $i<j$. This finishes the verification if all involved faces of $(\gamma,e)$  are vertical. Otherwise, \Cref{lem:meaning of exit index} and Diagram \eqref{eq:diagram of a face} imply the statement; in any of the cases where the case excluded in \Cref{lem:meaning of exit index} is involved, the face in question is low. We will give this argument here once and will not repeat it in the verification of the remaining simplicial identities below. 
    
    Consider the diagram
    \begin{equation}\label{eq:simplicial composition diagram}
        \begin{tikzcd}[row sep=large]
        \Delta[{k-2}]\ar[d,hook,bend right,"{\mcal{S}_{\flat'}}"']\ar[r,hook,"{\partial_i}"] & \Delta[{k-1}]\ar[d,hook,bend right,"{\mcal{S}_{\flat}}"']\ar[r,hook,"{\partial_j}"] & \Delta[{k}]\ar[d,hook,bend right,"{\mcal{S}_{e}}"'] \ar[r,"{\gamma}"] & {N}\\
        \Delta[{1}]\times\Delta[{k-3}]\ar[u,two heads,"{\mcal{C}_{\flat'}}"',bend right] \ar[r,dashed] & \Delta[{1}]\times\Delta[{k-2}]\ar[u,two heads,"{\mcal{C}_{\flat}}"',bend right]\ar[r,dashed] & \Delta[{1}]\times\Delta[{k-1}]\ar[u,two heads,"{\mcal{C}_{e}}"',bend right]\ar[ur,dashed]
        \end{tikzcd}
        .
    \end{equation}
    Without loss of generality, say 
    \(
        d_i(d_j(\gamma))=(\partial_j\partial_i)^*\gamma
    \) 
    is low, so we need to show that so is $d_{j-1}d_i(\gamma)$. That $\mcal{S}_{\flat}\partial_{i}$ factors through $\{0\}\times\Delta[{k-2}]$ is equivalent to $\mcal{S}_e\partial_j\mcal{C}_{\flat}\mcal{S}_{\flat}\partial_{i}$ factoring thusly by \Cref{lem:meaning of exit index}. Now, \(\mcal{S}_e\partial_j\mcal{C}_{\flat}\mcal{S}_{\flat}\partial_{i}=\mcal{S}_e\partial_j\partial_i\) by the construction of $\mcal{C}_\flat$, and similarly \(\mcal{S}_e\partial_j\partial_i=\mcal{S}_e\partial_j\partial_i\mcal{C}_{\flat'}\mcal{S}_{\flat'}.\) Together with the same calculation for $\partial_i$ and $\partial_j$ replaced respectively by $\partial_{j-1}$ and $\partial_i$ in Diagram \eqref{eq:simplicial composition diagram}, we see that
    \begin{equation}\label{eq:simplicial factorisation check}
        \mcal{S}_e\partial_j\mcal{C}_{\flat}\mcal{S}_{\flat}\partial_{i}=\mcal{S}_e\partial_j\partial_i\mcal{C}_{\flat'}\mcal{S}_{\flat'} \quad \text{and} \quad \mcal{S}_{e}\partial_i\mcal{C}_{\flat}\mcal{S}_{\flat}\partial_{j-1}=\mcal{S}_{e}\partial_i\partial_{j-1}\mcal{C}_{\flat'}\mcal{S}_{\flat'}.
    \end{equation}
    The indices $\flat^{(')}$ in the two equations are a priori \emph{not} the same since they are calculated for different pairs of indices, but we just showed above in \Cref{eq:flat identity} that the primed flats on the right hand sides do coincide. Combined with the same simplicial identity for ${N}$, this means that the right hand sides in \eqref{eq:simplicial factorisation check} agree, which yields the statement. 

    \noindent\underline{$d_is_j=s_{j-1}d_i$ for $i<j$:} Similarly, we obtain
    \begin{equation}\label{eq:second simplicial identity}
        \flat^{k-1}_{\sharp^{k}_{e,j},i}=\sharp^{k-1}_{\flat^{k}_{e,i},j-1}    
    \end{equation}
    using 
    \begin{equation}\label{eq:formula for sharp}
        \sharp^{k}_{e,j}=
        \begin{cases}
            e, & j\geq e\\
            e+1, & j<e.
        \end{cases}
    \end{equation}
    Now, \Cref{lem:meaning of exit index} and Diagrams \eqref{eq:diagram of a face} and \eqref{eq:simplicial composition diagram} (mutatis mutandis; e.g., using \eqref{eq:second simplicial identity} instead of \eqref{eq:flat identity} for \eqref{eq:simplicial factorisation check}) again finish the verification, analogously to the above. We no longer  mention this below.

    \noindent\underline{$d_is_j=\id$ for $i=j$ or $i=j+1$:} This translates to the identity 
    \(
        \flat^{k-1}_{\sharp^k_{e,j},i}=e.
    \)

    \noindent\underline{$d_is_j=s_jd_{i-1}$ for $i>j+1$:} This translates to the identity
    \(
        \flat^{k-1}_{\sharp^{k}_{e,j},i}=\sharp^{k-1}_{\flat^{k}_{e,i-1},j}.
    \)

    \noindent\underline{$s_is_j=s_{j+1}s_i$ for $i\leq j$:} This translates to the identity
    \(
        \sharp^{k-1}_{\sharp^{k}_{e,j},i}=\sharp^{k-1}_{\sharp^{k}_{e,i},j+1}.
    \)
\end{proof}

\begin{theorem}\label{thm:exit path category of linked space}
    If ${M},{L},{N}$ are $\infty$-categories, $\pi\colon{L}\rightarrow{M}$ is a right fibration, and $\iota\colon{L}\rightarrow{N}$ is a cofibration, then $\EEx(\pi,\iota)$ is an $\infty$-category.
\end{theorem}
\begin{proof}
We will directly verify the weak Kan property. We call a sub-simplex $(\Delta[\ell]\hookrightarrow\Lambda^{k}_{i}\to\EEx)\in\EEx_{\ell}$ of a horn \emph{low} if it is in ${M}_{\ell}$, \emph{vertical} if in $\Phat_{\ell-1}$, and \emph{upper} if in ${N}_{\ell}$. Sometimes we will not distinguish ${L}$ from $\iota({L})$ in notation.

Let
\(
    h\colon\Lambda^n_{i}\to\EEx
\)
be an inner horn, i.e., $0<i<n$, which does not factor through $M$ or $N$. Then $h$ either has a low face, an upper face, or all of its faces are vertical. 

Suppose that $h$ has a low face. Then the low face is necessarily $h|_{0\dots n-1}\in{M}_{n-1}$; $h|_{n}\in{N}_0$ is the sole upper vertex, and all other faces are vertical. Let us write $h|_{\widehat{j}}$ for \(h|_{0\dots\widehat{j}\dots n}=h\circ \partial_j\colon\Delta[n-1]\to\EEx\) whenever $j\neq i$. Each  vertical face 
    \(
        h|_{\widehat{k}}\in\Phat_{n-2}\subset\EEx_{n-1}\) for  \(i\neq k< n
    \)
    has its $(n-2)$-face $h|_{\widehat{k}\widehat{n}}$ in common with the low face $h|_{\widehat{n}}$, which therefore gives a lift 
    \(
        h_{\widehat{k}\widehat{n}}\in{L}_{n-2}
    \)
    to ${L}$ thereof. Here we wrote 
    \(
        h|_{\widehat{k}}=(h_{\widehat{k}},e)
    \)
    and $(h_{\widehat{k}})_{\widehat{n}}=\iota(h_{\widehat{k}\widehat{n}})$. As each $h|_{\widehat{k}}$ itself has a low face, its exit index is necessarily maximal, i.e., 
    \(e=n-1.\)
    Now, we obtain the intermediate lifting problem 
    \begin{equation}\label{3900AFH}
        \begin{tikzcd}
        \Lambda^{n-1}_{i}\ar[d,hook]\ar[r,"{\bigcup_{i\neq k\in[n-1]}h|_{\widehat{k}\widehat{n}}}"] &[5em] {L}\ar[d,"{\pi}"]\\
        \Delta[n-1]\ar[ur,dotted,"{H_{\widehat{n}}}"description]\ar[r,"{h|_{\widehat{n}}}"'] & {M}
        \end{tikzcd}
    \end{equation}
    with solution $H_{\widehat{n}}$. (It is imperative here that $\pi$ be a right fibration and not merely an inner one since $i=n-1$ is allowed.) This yields the horn 
    \[
        \iota(H_{\widehat{n}})\cup\bigcup_{i\neq k\in[n-1]}\iota(h|_{\widehat{k}\widehat{n}})\colon \Lambda^n_{i}\to{N}
    \]
    which has a filler $H\in{N}_{n}$. 
    
    We now claim that $(H,n)$ fills $h$. The restriction of \(H\circ \mcal{C}_{n}\colon\Delta[1]\times\Delta[n-1]\to{N}\) to $\{0\}\times\Delta[2]$ is $\iota(H_{\widehat{n}})$, which factors through ${L}$ by construction. Further, \(\mcal{S}_{n}\circ\partial_n\colon[n-1]\to[1]\times[n-1]\) sends $n>j\mapsto(0,j)$, so $d_{n}H$ is low, hence 
    \(
        d_{n}(H,n)=\pi(H_{\widehat{n}})=h|_{\widehat{n}},
    \)
    as desired. Finally, if $k<n$, then $\mcal{S}_n\circ\partial_j$ hits both $\{0\}\times[n-1]$ and $\{1\}\times[n-1]$, so $d_{k}H$ is vertical. Since $\flat_{n,k}=n-1$ for $k<n$, we obtain 
    \(
    d_k(H,n)=(h_{\widehat{k}},n-1)=h|_{\widehat{k}},
    \)
    also as desired.

    Suppose now that $h$ has an upper face. Then it is necessarily $h|_{\widehat{0}}\in{N}_{n-1}$; $h|_{0}\in{M}_0$ is the sole low vertex, and all other faces $h|_{\widehat{k}}=(h_{\widehat{k}},1)\in\Phat_{n-2}$ are vertical with exit index necessarily minimal. 
    Now, $h$ is given by a horn $\widetilde{h}\colon\Lambda^{n}_{i}\to{N}$ with $\widetilde{h}|_{0}\in\iota({L}_0)$. Taking a filler $H$ of $\widetilde{h}$, we claim that $(H,1)$ fills $h$. The restriction of $\mcal{C}_{1}\colon[1]\times[n-1]\to[n]$ to $\{0\}\times[n-1]$ hits only $0$, so $H\circ\mcal{C}_1$ factors through the mapping cocylinder by construction, independently of the choice of filler $H$. Further, $\mcal{S}_1\colon[n]\to[1]\times[n-1]$ sends only $0$ to $\{0\}\times[n-1]$ while $\mcal{S}_1\circ\partial_0$ factors through $\{1\}\times[n-1]$. This means $d_{0}H$ is upper, so \(d_{0}(H,1)=h|_{\widehat{0}},\) as desired. Finally, \(d_{k}(H,1)=(d_{k}H,\flat_{1,k})=(h_{\widehat{k}},1)=h|_{\widehat{k}}\) for every $k\geq 1$, also as desired.

    Suppose, finally, that all faces of $h$ are vertical. Then $h|_{0}\in{M}_0$ is low and $h|_{n}\in{N}_0$ upper, and moreover there must exist an index \(1\leq e\leq n\) such that 
        $h|_{j}\in{M}_0$ for $j<e$ and $h|_{j}\in{N}_0$ for $j\geq e$
    as otherwise there would exist a pair $0<j<j'<n$ such that $h|_{j}\in{N}_{0}$ while $h|_{j'}\in{M}_0$, which is absurd since the edge $h|_{jj'}$ would be of type ${N}\to{M}$. Moreover, $e=1$ resp.\ $e=n$ are impossible, since then $h|_{\widehat{0}}$ resp.\ $h|_{\widehat{n}}$ would be low resp.\ upper. We have obtained
    \( 1<e<n.\)
    Now, we claim that the exit indices of the faces $h|_{\widehat{j}}\in\Phat_{n-2}$, $j\neq i$, are determined by $e$ via the formula
    \begin{equation}\label{AFNUJCI}
       h|_{\widehat{j}}=\begin{cases}
        (h_{\widehat{j}},e), & j\geq e,\\
        (h_{\widehat{j}},e-1), & j<e.
       \end{cases} 
    \end{equation}
    Indeed, that \(\mcal{C}^{n-1}_{\ell}(\{0\}\times[n-2])=\{0,\dots,\ell-1\}\) for any $1\leq \ell\leq n-1$ implies that if $j\geq e$, then $h_{\widehat{j}}\circ\mcal{C}^{n-1}_{e}$ factors through $\mscc{P}$, as does $h|_{\widehat{j}}\circ \mcal{C}^{n-1}_{e-1}$ if $j<e$. Conversely, suppose $h|_{\widehat{j}}$ has index $e'$: $(h|_{\widehat{j}})|_{0,\dots,e-1}$ must be low, which implies, by the definition of $\mcal{S}_{e'}$, that $e'\geq e$, and since there are no further low vertices, we have $e'\leq e$ and thus $e=e'$.

    Now, as $h$ is underlied by a horn $\widetilde{h}\colon \Lambda^{n}_i\to{N}$, we may choose a filler $H\in{N}_n$. We claim that $(H,e)$ fills $h$. In order to ensure that $H\circ \mcal{C}_{e}\colon\Delta[1]\times\Delta[n-1]\to{N}$ factors through $\mscc{P}$, it suffices to observe that the missing face $h|_{\widehat{i}}$ cannot be low for then the choice of filler $H$ does not affect the factorisation property in that $\widetilde{h}$ needs filling only away from $\iota({L})$. Indeed, the only such case would be when $i=n$, but $h$ is inner.
    Finally, we check that the exit indices of the faces of $H$ are correct: since $1<e<n$, no face of $H$ is low or upper, and \eqref{eq:formula for flat} implies $d_{j}(H,e)=h|_{\widehat{j}}$ due to \eqref{AFNUJCI}, as desired.
\end{proof}

\begin{definition}\label{4GGBWT5}
    We call a span ${M}\xleftarrow{\pi}{L}\xrightarrow{\iota}{N}$ of $\infty$-groupoids or $\infty$-categories, with $\pi$ a right fibration, and $\iota$ a cofibration, a \emph{linked $\infty$-groupoid} or \emph{linked space} resp.\ \emph{linked $\infty$-category}, of \emph{depth $1$}. We call $\EEx$ its \emph{exit path $\infty$-category}.
\end{definition}

\begin{remark}\label{SVLV6J2}
    Every EPC is naturally equipped with a conservative functor to $\mrm{N}([1])$.
\end{remark}

\begin{remark}
    For $\mfrak{S}=(M\leftarrow L\to N)$ a linked $\infty$-category, the core $\EEx(\mfrak{S})^{\simeq}$ of its exit path $\infty$-category is canonically isomorphic to $M^{\simeq}\amalg N^{\simeq}$.
\end{remark}

\begin{remark}[arbitrary depth-$1$ posets]\label{MPD0BDL}
    \Cref{4GGBWT5} admits a straightforward generalisation to linked $\infty$-categories indexed over a poset $P$ of depth $1$. More precisely, let $\mfrak{S}\to P$ denote a collection consisting of an $\infty$-category ${M}_p$ for each $p\in P$, an $\infty$-category ${L}_{p\prec q}$ for every non-trivial arrow $p\prec q$, and right fibrations ${L}_{p\prec q}\to {M}_p$ and cofibrations ${L}_{p\prec q}\to{M}_q$. Then, \Cref{dfn:exit of span} and \Cref{thm:exit path category of linked space} generalise in the evident way to yield an $\infty$-category $\EEx(\mfrak{S}\to P)$. In the rest of this paper we will mostly consider spans and leave this level of generality implicit. See \Cref{GPK27YP} for a systematic treatment. In the language of \cite{barwick2020exodromy}, $\mfrak{S}\to P$ is a quasi-categorical d\'ecollage over $P$.
\end{remark}

\begin{remark}[arbitrary $\iota$]\label{3F1JHL6}
    In a homotopically-stratified set \`a la Quinn \cite{quinn1988homotopically}, links are defined as path spaces, in which case the link map $\iota$, the target evaluation, is typically not a cofibration. Generally, one may simply wish to adjoin a path space as the space of non-invertible paths between two otherwise unrelated spaces or categories. We will now discuss a variation on \Cref{dfn:exit of span} that can handle such input. 
    
    Let $\mfrak{S}=({M}\xleftarrow{\pi}{L}\xrightarrow{\iota}{N})$ be a span of $\infty$-categories where $\pi$ is a right fibration, and suppose $k\geq1$ and $j\in\{1,\dots,k\}$. Given $(\gamma,j)\in\Phat_{k-1}$, we have that $\mbf{b}(\gamma,j)\coloneqq\gamma|_{0,\dots,j-1}$, its \emph{maximal low face} or \emph{base}, lifts, albeit non-uniquely, to ${L}$. Let us now set 
    \[
        \widehat{\Phat_{k-1}}\coloneqq\coprod_{j=1}^{k}{L}_{j-1}\times_{{N}_{j-1}}\Phat_{k-1,j},
    \]
    where $\Phat_{k-1,j}\subset{N}_{k}$ denotes the set of exit $k$-simplices of index $j$, the map $\Phat_{k-1,j}\to{N}_{j-1}$ is given by $\mbf{b}$, and ${L}_{j-1}\to{N}_{j-1}$ by $\iota$. 

    There is a simplicial set $\EEx=\EEx(\mfrak{S})$ whose simplices are given by $\EEx_0={M}_0\amalg{N}_0$, and $\EEx_{k}={M}_k\amalg\widehat{\Phat_{k-1}}\amalg{N}_{k}$ for $k\geq1$. 
    We call the face of a simplex that is not wholly within either stratum \emph{low}, \emph{upper}, or \emph{vertical} by referring to its second coordinate. 
    Now, let $(\widehat{\mbf{b}},\gamma,j)\in\widehat{\Phat_{k-1}}$. If $d_i\gamma$ is low, then we set $d_i(\widehat{\mbf{b}},\gamma,j)=\pi(\widehat{\mbf{b}})$; if it is upper, then $d_i(\widehat{\mbf{b}},\gamma,j)=(d_i\gamma)$; and if it is vertical, then $d_i(\widehat{\mbf{b}},\gamma,j)=(\widehat{\mbf{b}}|_{\widehat{i}},d_i\gamma,\flat_{j,i})$, noting that $\mbf{b}(d_i\gamma,\flat_{j,i})=d_i\gamma|_{0,\dots,\flat-1}=\mbf{b}|_{\widehat{i}}$. Here, $\mbf{b}|_{\widehat{i}}$ is $\mbf{b}$ if $i> \flat-1$ and $d_i\mbf{b}$ if $i\leq\flat-1$, that is, the restriction of $\mbf{b}$ along $[\flat-1]\hookrightarrow[k-1]\overset{\partial_i}{\hookrightarrow}[k]$; similarly for $\widehat{\mbf{b}}$. Finally, we set $s_i(\widehat{\mbf{b}},\gamma,i)=(\widehat{\mbf{b}}_{i^+},s_i\gamma,\sharp_{j,i})$, where $\widehat{\mbf{b}}_{i^+}$ is the pullback of $\widehat{\mbf{b}}$ along $[\sharp_{j,i}]\hookrightarrow[k+1]\overset{\sigma_i}{\twoheadrightarrow}[k]$. 
    
    The proofs of \Cref{Q9W21WG} and \Cref{thm:exit path category of linked space} apply mutatis mutandis to prove that $\EEx$ is an $\infty$-category. 
    The further generalisation in the situation of \Cref{MPD0BDL} is immediate.
    Finally, if $\iota$ is a cofibration, the resulting $\infty$-category is canonically isomorphic to the $\EEx$ of \Cref{dfn:exit of span}.
\end{remark}

\begin{remark}\label{0K44G88}
    It is possible to define exit paths without referring to shuffles. Namely, saying $(\gamma,j)\in\Phat_{k-1}$ is equivalent to asking that $\gamma\in{N}_k$ and that $\mbf{b}(\gamma,j)\coloneqq\gamma|_{0,\dots,j-1}$ lifts to ${L}$. The face $d_i\gamma$ being low is equivalent to $d_i\gamma\subset\mbf{b}(\gamma,j)$, and it being upper to $d_i\gamma\subset\gamma|_{j,\dots,k}$. Now we can \emph{define} $\flat$ and $\sharp$ by the formulas \eqref{eq:formula for flat} and \eqref{eq:formula for sharp}. Using \Cref{dfn:exit of span} verbatim yields an equivalent definition of $\EEx$. This version, less motivated and more combinatorial, is easier to formulate, but would make some proofs below more challenging. Shuffles will prove especially useful in later sections.
\end{remark}

\subsection{Homotopy categories}\label{V3YDWXA}

Recall that for $S$ a simplicial set, and $h\in S_2$, $f\in S_1$, $g\in S_1$ such that $f,g\colon x\to y$ for vertices $x,y\in S_0$, the $2$-simplex $h$ is called a homotopy from $f$ to $g$ if $d_0h=\id_y$, $d_1h=g$ and $d_2h=f$. 
If $S=\mscc{C}$ is an $\infty$-category, the existence of a homotopy from $f$ to $g$ is equivalent to the existence of a $2$-simplex $h$ such that $d_0h=f$, $d_1h=g$ and $d_2h=\id_x$:
\ncdon \label{LYGOAV0}
    & x\ar[dr,"{f}"] & \\
    x \ar[ur,"{\id}"]\ar[rr,"{g}"] && y
\ncdoff
This latter version affords us, by \Cref{thm:exit path category of linked space}, the following useful fact:
\begin{lemma}\label{9SE20VL}
    Let $\mscc{C}=\EEx({M}\xleftarrow{\pi}{L}\xrightarrow{\iota}{N})$ be the exit path $\infty$-category of a linked $\infty$-category and let $x\in{M}$, $y\in{N}$. Then an index-$2$ exit $2$-path $h=(\widetilde{h},2)\in\EEx_2$  satisfies $d^{{N}}_2\widetilde{h}\in\iota({L}_{x})$ witnesses a homotopy between $d_0h$ and $d_1h$. Conversely, every homotopy can be witnessed by such an exit $2$-path.
\end{lemma}

\Cref{EY7G1PB} below shows that the space $\Hom_{\EEx}(p,q)$ of paths from $p\in{M}$ to $q\in{N}$ is equivalent to the space $\mscc{P}_{{L}_{p},q}$ of paths in ${N}$ that start in the embedded fibre $\iota({L}_{p})$ and end in $q\in{N}$.
Thus \Cref{9SE20VL} and \Cref{EY7G1PB} combine to imply a fibrewise gauge invariance (where we suppress $\iota$ and write $\circ$ to denote any composition):

\begin{lemma}\label{WLEZ9P7}
    Let $\mscc{C}=\EEx({M}\xleftarrow{\pi}{L}\xrightarrow{\iota}{N})$ be the exit path $\infty$-category of a linked $\infty$-category and let $x\in{M}$, $y\in{N}$. Then if $f=(\widetilde{f},1)\colon x\to y$ is an exit path with $\widetilde{f}(0)=l\in {L}_x$ and $\gamma\colon l'\to l$ is a path in ${L}_x$, then  $f$ is homotopic to the exit path $g=(f\circ\gamma,1)\colon x\to y$.
\end{lemma}
\begin{proof}
    Take $h=(\widetilde{h},2)$ as in \eqref{LYGOAV0} where $\widetilde{h}$ is a filler for $d_2^{{N}}\widetilde{h}=\gamma$ and $d_0^{N}\widetilde{h}=\widetilde{f}$.
\end{proof}

\begin{corollary}\label{AOSITE5}
    Let $\mfrak{S}=(M\xleftarrow{\pi} L\to N)$ be a linked space, let $p\in M$, $q\in N$ be points, and let $\Pi_1=\Pi_1(\mfrak{S};\{p,q\})$ be the homotopy category of the full sub-$\infty$-category of $\EEx(\mfrak{S})$ generated by $p$ and $q$. If $M$ and $N$ are path-connected, then the isomorphism class of the category $\Pi_1$ is independent of the choice of $p$ and $q$.
\end{corollary}

\begin{proof}
    We will write square brackets to indicate homotopy classes and continue suppressing $\iota$. Let $\Gamma\colon p\to p'$ be a path in $M$ and let $\Delta\colon q\to q'$ be a path in $N$. These induce the classical bijections $[\mrm{ad}_{\Gamma}]\colon\mrm{End}(p)\to\mrm{End}(p')$ and $[\mrm{ad}_{\Delta}]\colon\mrm{End}(q)\to\mrm{End}(q')$ between the respective endomorphism sets in the categories $\Pi_1(\mfrak{S};\{p,q\})$ and $\Pi_1(\mfrak{S};\{p',q'\})$. Without loss of generality, we may assume $q=q'$ and suppress $\Delta$. In order to construct a map $\Phi\colon\Hom(p,q)\to \Hom(p',q)$, consider, as in the classical proof that the fibres of $\pi$ are homotopy equivalent, the map $L_p\times I\to M$, $(l,t)\mapsto \Gamma(t)$, and fix a lift $H\colon L_p\times I\to L$ along the fibre inclusion $L_p\times\{0\}\to L$. (Recall that $H(-,1)\colon L_p\to L_{p'}$ is a homotopy equivalence.) Now, let $\alpha\in \Hom(p,q)$ and let a representative $(\widetilde{\alpha},1)\in \EEx_1$ be given. Set 
    \[
        \Phi(\alpha)=[(H(\widetilde{\alpha}(0),-)^{-1}\ast\widetilde{\alpha},1)].
    \]
    By \Cref{WLEZ9P7} on $\Hom(p,q)$, all representatives of $\alpha$, up to homotopy in $N$, are of type $(\epsilon\ast\widetilde{\alpha},1)$ where $\epsilon\colon l\to \widetilde{\alpha}(0)$ is a path in $L_p$. In order to show well-definedness, it suffices, by \Cref{WLEZ9P7} on $\Hom(p',q)$, to find a homotopy $H(l,-)^{-1}\ast\epsilon\sim\epsilon'\ast H(\widetilde{\alpha}(0),-)^{-1}$ in $N$ for some path $\epsilon'\colon H(l,1)\to H(\widetilde{\alpha}(0),1)$ in $L_{p'}$, but this is provided by taking $\epsilon'=H(\epsilon(-),1)$ since $H(\epsilon(-),-)$ provides such a homotopy. Thus, 
    \(
        [(H(l,-)^{-1}\ast\epsilon\ast\widetilde{\alpha},1)]=[(\epsilon'\ast H(\widetilde{\alpha}(0),-)\ast\widetilde{\alpha},1)]=[(H(\widetilde{\alpha}(0),-)^{-1}\ast\widetilde{\alpha},1)]
    \)
    and so $\Phi$ is well-defined. 

    It is clear that the resulting map $\Phi\colon\Pi_1(\mfrak{S};\{p,q\})\to\Pi_1(\mfrak{S};\{p',q\})$ is functorial with respect to post-composition with loops at $q$. Let now $\theta\colon p\to p$ be a loop at $p$ and let $\Theta\colon l\to \widetilde{\alpha}(0)$ in $L$ be a lift so that the composition $\alpha\circ[\theta]$ is represented by the exit path $(\Theta\ast\widetilde{\alpha},1)$ and we have $\Phi(\alpha\circ[\theta])=[(H(l,-)^{-1}\ast\Theta\ast\widetilde{\alpha},1)]$. On the other hand, $\Phi(\alpha)\circ\Phi([\theta])=[(H(\widetilde{\alpha}(0),-)^{-1}\ast\widetilde{\alpha},1)]\circ[\mrm{ad}_{\Gamma}(\theta)]$, in order to represent which let $\mrm{AD}_{\Gamma}(\theta)\colon l'\to H(\widetilde{\alpha}(0),1)$ in $L$ be a lift of $\mrm{ad}_{\Gamma}(\theta)$. Then $\Phi(\alpha)\circ\Phi([\theta])=[(\mrm{AD}_\Gamma(\theta)\ast H(\widetilde{\alpha},-)^{-1}\ast\widetilde{\alpha},1)]$. But we may take 
    \(
        \mrm{AD}_\Gamma(\theta)=H(l,-)^{-1}\ast\Theta\ast H(\widetilde{\alpha}(0),-)
    \)
    which gives $[(H(l,-)^{-1})\ast\Theta\ast H(\widetilde{\alpha}(0),-)\ast H(\widetilde{\alpha}(0),-)^{-1}\ast\widetilde{\alpha},1)]=\Phi(\alpha\circ[\theta])$, proving the functoriality of $\Phi\colon\Pi_1(\mfrak{S};\{p,q\})\to\Pi_1(\mfrak{S};\{p',q\})$.

    In order to show that $\Phi$ is an isomorphism, let $K\colon L_{p'}\times I\to L$ be a lift of $L_{p'}\times I\to M$, $(l,t)\mapsto\Gamma(1-t)$ along the fibre inclusion $L_{p'}\times \{0\}\to L$ and let $\Psi\colon\Pi_1(\mfrak{S};\{p',q\})\to \Pi_1(\mfrak{S};\{p,q\})$ be defined by $\Psi|_{\mrm{End}(p')}=[\mrm{ad}_{\Gamma^{-1}}]$ and $\Psi(\beta)=[(K(\widetilde{\beta}(0),-)^{-1}\ast\widetilde{\beta},1)]$ where $\beta=[(\widetilde{\beta},1)]\in\Hom(p',q)$. We observe $\Psi\Phi(\alpha)=[(K(H(\widetilde{\alpha}(0),1),-)^{-1}\ast H(\widetilde{\alpha}(0),-)^{-1}\ast\widetilde{\alpha},1)]$, but the first concatenation is over $\Gamma\ast\Gamma^{-1}$, so let $I\times I\to M$ be a homotopy $\Gamma\ast\Gamma^{-1}\to\mrm{const}_p$ and $\Xi\colon I\times I\to L$ be a homotopy extension of $K(H(\widetilde{\alpha}(0),1),-)^{-1}\ast H(\widetilde{\alpha}(0),-)^{-1}\colon I\times\{0\}\to L$. Then the latter is homotopic to $\Xi(-,1)$, a path contained within $L_p$, hence $\Psi\Phi(\alpha)=[(\Xi(-,1)\ast \widetilde{\alpha},1)]=[(\widetilde{\alpha},1)]=\alpha$ using \Cref{WLEZ9P7}. The converse $\Phi\Psi=\id$ is analogous.
\end{proof}

The bijections $\Hom(p,q)\to\Hom(p',q')$ in the context of \Cref{AOSITE5} need not be given by conjugation, as \Cref{BAI62S9} illustrates. We note the following simple description:

\begin{lemma}\label{5QWZ624}
    Let $\Pi_1$ be as in \Cref{AOSITE5}, and suppose $q\in L_p$ and that $L_p$ is connected. Then $\Pi_1$ is a category with objects $p$ and $q$, satisfying  $\mrm{End}(p)\cong\pi_1(M)$, $\mrm{End}(q)\cong\pi_1(N)$, and 
    \(
        \Hom(p,q)\cong\pi_1(N)/\pi_1(L_p).
    \)
\end{lemma}
\begin{proof}
    This follows from \Cref{EY7G1PB} and the homotopy long exact sequence of the path space fibration $\evv_1\colon P_{L_p,-}\to N$. 
\end{proof}

\subsection{Examples}\label{8PVG15K}

    Any right fibration $\pi$ alone, or any cofibration $\iota$ alone gives an example with a trivial choice for the other leg of the span: the identity cofibration 
    \(
        {M}\xleftarrow{\pi}{L}\xrightarrow{\id}{L}
    \)
    or the final fibration to the point
    \(
        \ast\leftarrow{L}\xrightarrow{\iota}{N}.
    \)
    In the latter case, there is one more trivial option: 
    \(
        \partial_{\iota}=({L}\xleftarrow{\id}{L}\xrightarrow{\iota}{N}).
    \)
    In view of \Cref{ex:bordisms} below, we interpret $\partial_\iota$ as a \emph{linked bordism} with boundary ${L}$ and interior ${N}$. Not all linked bordisms arise from ordinary bordisms, e.g.\ if $\dim L<\dim M -1$. In view of \Cref{3F1JHL6}, $\iota$ need not even be a cofibration. In a similar vein, any $\infty$-category ${X}$ gives a linked $\infty$-category 
    \(
        \emptyset\leftarrow\emptyset\to{X}
    \)
    with $\EEx(\emptyset\leftarrow\emptyset\to{X})\cong{X}$. The other trivial construction 
    \(
        \ast\leftarrow{X}\xrightarrow{\id}{X}
    \)
    corresponds to taking the open cone of ${X}$ in the sense that 
    \(
        \ast\in\EEx(\ast\leftarrow{X}\xrightarrow{\id}{X})\simeq {X}^{\triangleleft}
    \)
    is initial. We note this in \Cref{3IZYV5S}.

\begin{example}\label{ex:bordisms} The linked space corresponding to a bordism $(M,\partial M)$ has lower stratum $\partial M$, higher stratum $M^\circ=M\smallsetminus\partial M$, link $L=\partial M$, $\pi=\id_{\partial M}$, and $\iota\colon\partial M\hookrightarrow M^\circ$ given by the flow along a nowhere-vanishing inward-pointing vector field along the boundary for a chosen non-zero time.
\end{example} 

\begin{example}\label{ex:defects} With a closed submanifold $\Sigma\subset M$ we may associate a linked space with lower stratum $\Sigma$, higher stratum $M\smallsetminus \Sigma$, and link $\mbb{S}\Normal (\Sigma)$, the sphere bundle of the normal bundle of $\Sigma$, together with the projection $\pi$ and the choice $\iota$ of a tubular neighbourhood.
\end{example}

\begin{example}\label{ex:grassmann}
    For $n,m\in\mbb{N}$, consider the linked space 
    \[
        B\OO(n,m)\coloneqq(B\OO(n)\twoheadleftarrow B\OO(n)\times B\OO(m)\xhookrightarrow{\oplus}B\OO(n+m)),
    \] 
    the (infinite) \emph{$(n,m)$-Grassmannian}, where $\oplus$ is induced by taking direct sums of vector spaces and the choice of an isomorphism $\R^\infty\oplus\R^\infty\cong\R^\infty$. Similar linked spaces exist for classical structure groups. It is shown in the companion paper \cite{tetik2022stratified} that $\EEx(B\OO(n,m))$ embeds as a full sub-$\infty$-category into a strictified quasi-category model of the stratified Grassmannian of Ayala--Francis--Rozenblyum \cite{ayala2020factorization}. Analogously, the Stiefel manifolds of the same ranks assemble into a linked space $E\OO(n,m)$ over $B\OO(n,m)$, and $\EEx(E\OO(n,m))$ embeds into a similar stratified Stiefel $\infty$-category. Analogous statements hold for classical structure groups.
\end{example}

Using \Cref{5QWZ624} we will calculate the homotopy categories of some linked spaces.

\begin{example}\label{7RNKPF5}
    Consider $\mbb{C}_{0}=(\{0\}\leftarrow S^1\to\mbb{C}^\ast)$, the linked space corresponding to the complex plane stratified by the origin. Let $\Pi_{1}(\mbb{C}_{0};\{0,i\})$ denote the full sub-category of the homotopy category $\mrm{h}\EEx(\mbb{C}_{0})$ generated by the objects $0$ and $i$. Then $\mscc{P}_{S^1,i}$ is contractible, hence $\Hom_{\Pi_{1}(\mfrak{S};\{0,i\})}(0,i)=\ast$. Consequently, we obtain an isomorphism $\Pi_{1}(\mbb{C}_{0};\{0,i\})\cong (\mrm{B}\mbb{Z})^{\triangleleft}$.

    Similarly, let $(\mbb{C}\smallsetminus\{1\})_{0}=(\{0\leftarrow S^1\to\mbb{C}\smallsetminus\{0,1\}\})$, let $\mbb{Z}\ast\mbb{Z}$ denote the free product of $\mbb{Z}$ with itself, and finally let $e\colon\mbb{Z}\hookrightarrow\mbb{Z}\ast\mbb{Z}$ denote one of the two subgroup inclusions. In the homotopy long exact sequence of the path space fibration $\evv_1\colon\mscc{P}_{S^1,-}\to\mbb{C}\smallsetminus\{0,1\}$, the map $\mbb{Z}\cong\pi_1(\mscc{P}_{S^1,-})\to\pi_1(\mbb{C}\smallsetminus\{0,1\})\cong\mbb{Z}\ast\mbb{Z}$ is given by $\pi_1(\evv_1)=e$, which implies $\pi_0(\mscc{P}_{S^1,i})\cong(\mbb{Z}\ast\mbb{Z})/e$. This is a mere set since $\mrm{Im}(e)$ is not normal. We conclude that the two-object category $\Pi_{1}((\mbb{C}\smallsetminus\{1\})_{0};\{0,i\})$ satisfies $\mrm{End}(0)=\ast$, $\mrm{End}(i)=\mbb{Z}\ast\mbb{Z}$, $\Hom(0,i)=(\mbb{Z}\ast\mbb{Z})/\mbb{Z}$, and the composition $\mrm{End}(i)\times\Hom(0,i)\to\Hom(0,i)$ is given by $(w,w'\mbb{Z})\mapsto ww'\mbb{Z}$.
\end{example}

\begin{example}\label{BAI62S9}
    Let $K\colon S^1\hookrightarrow S^3$ be a knot with knot group $W=\pi_1(S^3\smallsetminus K)$. Since the normal bundle $\nu$ of any knot is trivialisable, the linked space which \Cref{ex:defects} associates with $K$ can be taken to be $S^3_{K}=(S^1\leftarrow S^1\times S^1\to S^3\smallsetminus K)$ where the link embedding $\iota\colon S^1\times S^1\to S^3\smallsetminus K$ is essentially determined by $K$ and the normal framing. Let now $p\in S^1$ and $q\in S^3\smallsetminus K$, and note $L_p=S^1$. Analogously to \Cref{7RNKPF5}, the target evaluation $\mscc{P}_{S^1,-}\to S^3\smallsetminus K$ composed with the constant loop inclusion $S^1\simeq\mscc{P}_{S^1,-}$ induces a group homomorphism $e\colon \mbb{Z}\to W$ on fundamental groups and we obtain $\Hom_{\Pi_1(S^3_{K};\{p,q\})}(p,q)=\pi_0(\mscc{P}_{S^1,q})\cong W/e$. The rest of the two-object category $\Pi_1(S^3_K;\{p,q\})$ is then given similarly to the one in \Cref{7RNKPF5}. 
    
    Recalling the Wirtinger presentation of $W$, we observe that the map $S^1\to S^3\smallsetminus K$ inducing $e$ specifies a generator $\epsilon$ of $W$, so that $e$ maps onto the subgroup of $W$ generated by $\epsilon$. For instance, for $K=U$ the unknot, we have $W\cong \mbb{Z}$ and so $W/e\cong\ast$ for either normal framing (corresponding to the choice $\epsilon=\pm1$).

    If $K=T$ is the trefoil knot, then $W\cong B_3$ is the braid group on $3$ strands, one presentation of which is $\langle a,b : a^2=b^3\rangle$. Let $a_1, a_2, a_3\in W$ denote the three generators of $W$ so that $\epsilon=a_i^{\pm 1}$ for some $i\in\{1,2,3\}$. Let us assume $\epsilon=a_i$ upon changing the normal framing (or the plane projection) if necessary. An isomorphism $W\to B_3$ is given by $a_1\mapsto ab^{-2}aba^{-1}$, $a_2\mapsto ab^{-1}$, $a_3\mapsto b^2a^{-1}$.  Given this identification, \Cref{AOSITE5} implies that the quotients $B_3/a_{i}$ for all $i\in\{1,2,3\}$ are in bijection with one another. There are two cases depending on the choice of the basepoints $p$ and $q$: 
    \begin{itemize}
        \item For $\epsilon\in\{a_2,a_3\}$, $b^2a^{-1}$ and $ab^{-1}$ are conjugate, giving a bijection $B_3/b^2a^{-1}\cong B_3/ab^{-1}$.
        \item If $\epsilon=a_1=ab^{-2}aba^{-1}$, then $W/\epsilon\cong B_{3}/b^{-1}a$ upon conjugation. Although $ab^{-1}$ is not conjugate to $b^{-1}a$, a bijection $B_3/b^{-1}a\cong B_3/ab^{-1}$ can be given as follows: there is an isomorphism $R\colon F_2=\langle a,b\rangle\to F_2$ on the free group on two letters given by reversing the order of letters in words. Then $R$ induces a bijection $B_3/ab^{-1}\cong b^{-1}a\backslash B_3\cong B_3/b^{-1}a$.
    \end{itemize}
\end{example}

\begin{example}\label{02FAWCN}
    It is well known that $c=a^2=b^3$ in $B_3\cong\langle a,b : a^2=b^3\rangle$ generates the infinite cyclic centre of $B_3$ and that $B_3$ is the universal central extension of the modular group: $B_3/c\cong \mrm{PSL}(2;\mbb{Z})$. Now, diverging from the context of \Cref{ex:defects} slightly, consider a linked space of type $\mfrak{S}_{K,\iota}=(\ast\leftarrow S^1\xhookrightarrow{\iota}S^3\smallsetminus K)$ where $K$ is a knot. Taking $K=T$ to be the trefoil knot, $p=\ast$, $q\in S^3\smallsetminus\{K\cup\iota\}$  and $\iota\sim(a_3a_2)^{\ast 3}$ a smooth embedding representative of $c$ written in terms of the Wirtinger generators, we obtain $\Hom_{\Pi_1}(p,q)\cong \mrm{PSL}(2;\mbb{Z})$. Thus, the two-object category $\Pi_1$ is determined by the modular group together with the action of $B_3$ thereon. 
    Generalising slightly, any group quotient $G/H$, where $H$ is a (not necessarily normal) finitely generated  subgroup can be realised as a morphism set $\Hom_{\Pi_1}(\ast,q)$ by considering  $\ast\leftarrow (S^1)^{\vee k}\xhookrightarrow{\iota} K(G,1)$ for $\iota$ representing the generators of $H$.
\end{example}

\section{Morphism spaces and constant exit loops}\label{TCFKEBQ}

\begin{notation*}
    Given an embedding $\iota\colon \Sigma\hookrightarrow N$ and a point $q\in N$, we let 
    $
    P(N)_{\Sigma,q}=P_{\Sigma,q}
    $
    denote the space of paths in $N$ that start in $
    \iota(\Sigma)$ and end in the point $q$, equipped with the compact-open topology. We use analogous notation when we work with a cofibration $\iota$ of simplicial sets.
\end{notation*}

The following result formalises and confirms the intuition that the link represents an infinitesimal expansion of the lower stratum into the higher stratum, i.e., the boundary of its unzip (see \Cref{YSAL1EK}). More precisely, it is a pointwise version of that fact.

\begin{lemma}\label{EY7G1PB}
    Let $\mathfrak{S}=({M}\xleftarrow{\pi}{L}\xrightarrow{\iota}{N})$ be a linked $\infty$-category, and $p\in {M}$ and $q\in {N}$ points in the two strata. There is an equivalence
    \(
        \Hom_{\EEx(\mathfrak{S})}(p,q)\simeq \mscc{P}_{{L}_p,q}
    \)
    between the morphism space in $\EEx$ from $p$ to $q$ and the space of paths in ${N}$ that start in the embedded fibre $\iota({L}_p)$, where ${L}_p=\{p\}\times_{{M}}{L}$, and end in $q$.
\end{lemma}
\begin{proof}
    We will work with a model for morphism spaces that makes the proof particularly simple: by \cite[01L5]{kerodon}, the morphism space in $\EEx$ is equivalent to the right-pinched morphism space 
    \(
        \Hom^{\mathrm{R}}_{\EEx}(p,q)\coloneqq \{p\}\times_{\EEx}(\EEx/q).
    \)
    We will observe that $\{p\}\times_{\EEx}(\EEx/q)$ is isomorphic to $\mscc{P}_{{L}_p,q}={L}_p\times_{{N}}({N}/q)$.
    Indeed, at vertex level, the bijection
    \(
        \left(\{p\}\times_{\EEx}(\EEx/q)\right)_0\cong\left({L}_p\times_{{N}}({N}/q)\right)_0
    \)
    is clear: recalling that non-invertible $1$-morphisms in $\EEx$ are elements of $\Phat_0\subset{N}_1$, let $(\gamma,1)\in\Phat_0$. For $p=d_1^{\EEx}(\gamma,1)=\pi(d_1^{{N}}(\gamma))$ to hold, we must have $d_1^{{N}}(\gamma)\in\iota({L}_p)$. Similarly, $d_0^{\EEx}(\gamma,1)=\iota(d_0^{{N}}(\gamma))$, which yields the bijection.

    Let now $k>0$ and consider an exit $(k+1)$-path $(\gamma\colon\Delta[k]\star\Delta[0]\to{N},j)$ in $(\EEx/q)_k\subset\EEx_{k+1}$. 
    Asking that $(\gamma,j)$ be in $\{p\}\times_{\EEx}(\EEx/q)$ is equivalent to asking that (1) its ${N}$-face 
        \(
            \Delta[k]\hookrightarrow\Delta[k]\star\Delta[0]\xrightarrow{\gamma} {N},
        \)
        which is $d_{k+1}^{{N}}(\gamma)$ under the standard identification $\Delta[k]\star\Delta[0]\cong\Delta[k+1]$, is low, as by construction only then can the corresponding $\EEx$-face be in ${M}_k\subset\EEx_{k}$; and that (2) it lies in particular in $\iota({L}_p)$.

    Condition (1) implies that $d^{{N}}_{\ell}(\gamma)\in{N}_k$ is vertical for $\ell<k+1$ since all other faces contain $q\colon\Delta[0]\hookrightarrow\Delta[k]\star\Delta[0]$, whence they are not low; if some $d_{\ell}(\gamma)$ was upper, that would contradict the lowness of its unique common $(k-1)$-face with $d_{k+1}^{{N}}(\gamma)$. In fact, $(\gamma,j)$ has no upper $n$-face once $n>0$: given $\Delta[n]\hookrightarrow\Delta[k+1]$, there is necessarily a vertex in $d^{{N}}_{d+1}(\gamma)$ in its image.
    But then the exit index $j$ must be maximal:
    \(
    j=k+1.
    \)
    For if not, then there would exist at least one upper $n$-face for $n>0$, the largest such, with $n=k+1-j$, for instance, being specified by $[n]\hookrightarrow [k+1]$, $\alpha\mapsto \ell+\alpha$. We thus obtain a bijection
    \(
        \left(\{p\}\times_{\EEx}(\EEx/q)\right)_k\cong\left({L}_p\times_{{N}}({N}/q)\right)_k
    \)
    by reducing exit paths $(\gamma,j)$ on the left-hand side to those of index $k+1$, and so to only a subset of ${N}_{k+1}$, and specifically those such that $d_{k+1}^{{N}}(\gamma)\in{L}_p$. Finally, $\left(\{p\}\times_{\EEx}(\EEx/q)\right)_{\ast}\xrightarrow{}\left({L}_p\times_{{N}}({N}/q)\right)_\ast$, thus defined, is functorial: any vertical face of $(\gamma,k+1)$ is again of maximal index since, using \eqref{eq:formula for flat} and \eqref{eq:formula for sharp}, we have $\flat^{k+1}_{k+1,i}=k$ and, for degeneracies, we have $\sharp^{k+1}_{k+1,i}=k+2$ for all $i<k+1$. 
\end{proof}

\begin{corollary}\label{3IZYV5S}
    For a linked space of type 
    \(
        \ast\leftarrow {N}\overset{\id}{\to}{N}
    \)
    we have
        \(
        \Hom_{\EEx}(\ast,q)\simeq\ast.
        \)
        Consequently, when ${N}=\Sing_\bullet(N)$ for $N$ a smooth manifold, we have 
        \(
            \EEx\simeq\mathrm{Exit}(C(N)),
        \)
        where the left-hand side is the exit path $\infty$-category 
        of the open cone 
        \(
            C(N)=\ast\amalg_{\{0\}\times N}([0,1)\times N)
        \)
        on $N$ with its canonical stratification over $\{0<1\}$.
\end{corollary}
\begin{proof}
    We have $\mscc{P}_{{L}_\ast,q}\simeq N/q\simeq\ast$ since $N$ is an $\infty$-groupoid, so
    \(
        \EEx\simeq N^{\triangleleft}.
    \) 
    The second statement follows from \cite[Proposition 3.3.8]{ayala2017stratified}.\footnote{This is an equivalence of quasi-categories for Lurie's model from \cite{luriehigheralgebra}, or, after translating to the complete Segal space model and using \cite[Lemma 3.3.9]{ayala2017stratified}, with that of Ayala et al.}
\end{proof}

\begin{remark} 
    Since the link of the cone locus $\ast$ and the interior of $C(N)$ is $N$ itself, one could consider $\ast\twoheadleftarrow N\hookrightarrow N\times\R$ the natural linked space model for the open cone. The proof of \Cref{3IZYV5S} implies that this modification changes the exit path $\infty$-category only up to equivalence. More generally, homotopy equivalent linked spaces have equivalent EPCs: see \Cref{S2KANF2} for the topological category of linked spaces and \Cref{5X56IK6} for $\EEx$ as a topological functor into the Kan-enriched category of $\infty$-categories.
\end{remark}

    Next, we will generalise \Cref{EY7G1PB} in for linked spaces and identify $L$, up to equivalence, with the space of paths in $\EEx$ from $M$ to $N$.
    Namely, for $p\in M$ and $q\in N$, observe that  
    \(
        \down[N]{L_{p}}{q}=P_{L_{p},q}\simeq \Hom(p,q)=\down[\EEx]{p}{q}.
    \)
    Formally, varying $p$ should give an equivalence 
    \(
        \down[N]{L}{q}=P_{L,q}\simeq\down[\EEx]{L}{q}\simeq\down[\EEx]{M}{q}
    \)
    and then varying $q$ should give
    \(
        L\simeq\down[N]{L}{N}=P_{L}\simeq\down[\EEx]{M}{N}.
    \)

    This is indeed true, but we will start with the `constant exit loop inclusion' which is well-defined more generally for all linked $\infty$-categories. 

    \begin{lemma}\label{RPISW40}
        There is an $\infty$-functor
        \(
        \Box\colon L\to\Ar(\EEx)
        \)
        which factors through $(M\downarrow N)\subset\Ar(\EEx)$.
    \end{lemma}
    \begin{proof}
        Let $n\geq 0$ and $\gamma\in L_{n}$. We define the restriction of 
        \(
            \boxed{\gamma}\colon\Delta[1]\times\Delta[n]\to\EEx
        \)
        along $\Delta[n]\simeq\Delta\{1\}\times\Delta[n]\hookrightarrow\Delta[1]\times\Delta[n]$, its \emph{upper side}, to be
        \begin{equation}\label{5WPV6OL}
            \boxed{\gamma}_{\Delta\{1\}\times\Delta[n]}\coloneqq\iota(\gamma)\in M_{n}\subset \EEx_{n},
        \end{equation}
        and, its low side, the restriction along $\Delta[n]\simeq\Delta\{0\}\times\Delta[n]\hookrightarrow\Delta[1]\times\Delta[n]$ to be 
        \begin{equation}\label{HDCZQA6}
            \boxed{\gamma}_{\Delta\{0\}\times\Delta[n]}\coloneqq\pi(\gamma)\in N_n\subset\EEx_{n}.
        \end{equation}
        Further, given $i\in[n]$, we define the restriction along $\Delta[1]\simeq\Delta[1]\times\Delta\{i\}\hookrightarrow\Delta[1]\times\Delta[n]$ to be 
        \[\boxed{\gamma}_{\Delta[1]\times\Delta\{i\}}\coloneqq(s_0\iota(\gamma|_{i}),1)=(s_0\iota((\{i\}\hookrightarrow[n])^*\gamma),1)\in\Phat_0\subset\EEx_1.\] It is easily seen that these definitions are consistent.
    
        We define the restriction along the exit shuffle $\mcal{S}_j$, $j\in\{1,\dots,n+1\}$, to be 
        \begin{equation}\label{K0R63TJ}
            \boxed{\gamma}_{\mcal{S}_{j}}\coloneqq\left(s_{j-1}\iota(\gamma),j\right).
        \end{equation}
        Let us check that this is consistent with the definitions of $\boxed{\gamma}_{\Delta\{0/1\}\times\Delta[n]}$ and $\boxed{\gamma}_{\Delta[1]\times\Delta\{i\}}$ we have given. 
        We must have that the low part\footnote{This means the sub-simplicial set generated by the low vertices of the exit path.} of $\boxed{\gamma}_{\mcal{S}_{j}}$ coincides with the appropriate face of $\boxed{\gamma}_{\Delta\{0\}\times\Delta[n]}$. The low part is the restriction along the identity inclusion $\Delta[j-1]\hookrightarrow\Delta[n+1]$. Since $(\Delta[j-1]\hookrightarrow\Delta[n+1])^*s_{j-1}=(\sigma_{j-1}\circ([j-1]\hookrightarrow[n+1]))^*$ and since $[j-1]\hookrightarrow[n+1]\xrightarrow{\sigma_{j-1}}[n]$ coincides with the identity inclusion $[j-1]\hookrightarrow[n]$, we have that 
        \(
            (\Delta[j-1]\hookrightarrow\Delta[n+1])^{*}(s_{j-1}\iota(\gamma),j)=\pi(\gamma|_{0,\dots,j-1}),
        \) 
        as desired.
        Similarly, the upper part\footnote{generated by the upper vertices} is the restriction along $\Delta[n+1-j]=\Delta\{j,\dots,n+1\}\hookrightarrow\Delta[n+1]$. Since $\{j,\dots,n+1\}\hookrightarrow[n+1]\xrightarrow{\sigma_{j-1}}[n]$ is $\Delta[n+1-j]=\Delta\{j-1,\dots,n\}\hookrightarrow\Delta[n]$, we have 
        \(
            (\Delta\{j,\dots,n+1\}\hookrightarrow\Delta[n+1])^*(s_{j-1}\iota(\gamma),j)=\iota(\gamma|_{j-1,\dots,n}).
        \) 
        This is as desired, since $\mcal{S}_{j}(k)=(1,i-1)$ for $k\geq j$, so that the upper part as picked out in the simplex category $\mathbf{\Delta}$ by the image of $[n+1-j]\xhookrightarrow{+j}[n+1]\xhookrightarrow{\mcal{S}_{j}}[1]\times[n]$, which is  $\{1\}\times\{j-1,\dots,n\}$, is precisely $\boxed{\gamma}_{\{1\}\times\{j-1,\dots,n\}}=\iota(\gamma|_{j-1,\dots,n})$.  
        
        Finally, we must show that the restrictions along the $\mcal{S}_{j}$, thus defined, glue.
        Let a pair of distinct exit indices $j<j'$ in $\{1,\dots,n+1\}$ be given. The intersection of the images of $\mcal{S}_{j},\mcal{S}_{j'}\colon[n+1]\hookrightarrow[1]\times[n]$ as picked out within $\mathbf{\Delta}$ consists of a purely low part, $\{0\}\times[j-1]$, and a purely upper part, $\{1\}\times\{j'-1,\dots,n+1\}$. Let us write $\delta\coloneqq j'-j$ and consider the map 
        \[\mcal{S}_{j\cap j'}\colon \Delta[n+1-\delta]\hookrightarrow\Delta[1]\times\Delta[n]\]
        induced by the map $\mcal{S}_{j\cap j'}\colon[n+1-\delta]\hookrightarrow[1]\times[n]$ given by
        \[
        i\mapsto\begin{cases}
            (0,i), & i<j\\
            (1,i-1+\delta), & i\geq j.
        \end{cases}
        \]
        In other words, $\mcal{S}_{j\cap j'}$ is like $\mcal{S}_{j}$ except with upper part shifted by $j'-j$, and so that (the images in $[1]\times[n]$ of) its low and upper parts coincide precisely with those of $\mcal{S}_{j}$ and $\mcal{S}_{j'}$, respectively. That is, $\mcal{S}_{j\cap j'}$ picks out precisely the intersection of the images of $\mcal{S}_{j}$ and $\mcal{S}_{j'}$ within $\Delta[1]\times\Delta[n]$. Moreover, its factorisation through $\mcal{S}_{j}$ as well as through $\mcal{S}_{j'}$ is given by identifying 
        \[
            \Delta[n+1-\delta]=\Delta\{0,\dots,j-1,j',\dots,n+1\},
        \]
        in the sense that
        \cdon[column sep=large] 
            \Delta\{0,\dots,j-1,j',\dots,n+1\}
            \ar[d,hook,"{\Theta}"']
            \ar[r,hook,"{\mcal{S}_{j\cap j'}}"] 
        & 
            \Delta[1]\times\Delta[n]
        \\
            \Delta[n+1]
            \ar[ur,hook,"{\text{$\mcal{S}_{j}$ or $\mcal{S}_{j'}$}}"']
        \cdoff
        commutes. This can be checked within $\mathbf{\Delta}$: since $j<j'$, the restrictions of $\mcal{S}_{j}$ and $\mcal{S}_{j'}$ along $[0,\dots,j-1]\hookrightarrow[n+1]$ are both $i\mapsto (0,i)$, which coincides with $\mcal{S}_{j\cap j'}$. For $[n+1-\delta]\ni i\geq j$, we have $\Theta(i)=i+\delta\geq j'>j$ and so $\mcal{S}_{j}\Theta(i)=(1,\Theta(i)-1)=\mcal{S}_{j'}\Theta(i)$. This coincides with $\mcal{S}_{j\cap j'}(i)=(1,i-1+\delta)$, which proves the commutativity of the diagram.
    
        We must show, therefore, that the two maps
        \[
        \Theta^*\boxed{\gamma}_{\mcal{S}_{j}}, \Theta^*\boxed{\gamma}_{\mcal{S}_{j'}}\colon \Delta[n+1-\delta]\to\EEx 
        \]
        coincide. This will reduce to a fact about the behaviour of $\flat(-,-)\coloneqq\flat_{-,-}$. We observe  
        \begin{align*}
        \Theta&=\partial_{j'-1}\partial_{j'-2}\cdots\partial_{j+1}\partial_{j}\\
        &\colon\Delta[n+1-\delta]\hookrightarrow\Delta[n+1-\delta+1]\hookrightarrow\cdots\hookrightarrow\Delta[n+1]
        \end{align*}
        and note that $\Theta^*=d_{j}d_{j+1}\cdots d_{j'-2}d_{j'-1}=d_{j}d_{j}\cdots d_{j}d_{j}$ by repeated application of the simplicial identitiy $d_{\alpha}d_{\beta}=d_{\beta-1}d_{\alpha}$ for $\alpha<\beta$. 
        This implies 
        \begin{align*}
        \Theta^*\boxed{\gamma}_{\mcal{S}_{j}}&=(d_{j}d_{j}\cdots d_{j}\id\iota\gamma,\flat(\flat(\cdots\flat(j,j'-1),j+1),j))\\
        &=(d_{j}d_{j+1}\cdots d_{j'-3}d_{j'-2}\iota\gamma,j)
        \end{align*}
        using the simplicial identitiy $d_{j}s_{j-1}=\id$ and then by repeated un-application of the previous identity. For the exit index, we used that $\flat(\alpha,\beta)=\alpha$ if $\beta\geq\alpha$, so that 
        \begin{align*}
            \flat(\flat(\cdots\flat(j,j'-1),j+1),j)&=\flat(\flat(\dots\flat(j,j'-2),j+1),j)=\dots\\
            &=\flat(j,j)=j.
        \end{align*}
        On the other hand, using $\flat(\alpha,\beta)=\alpha-1$ if $\beta<\alpha$, we have 
        \begin{align*}
        \flat(\flat(\cdots\flat(j',j'-1),j+1),j)&=\flat(\flat(\cdots\flat(j'-1,j'-2),j+1),j)=\dots\\
        &=\flat(j+1,j)=j.
        \end{align*}
        Now, since $d_{j'-1}s_{j'-1}=\id$, we obtain 
        \begin{align*}
            \Theta^*\boxed{\gamma}_{\mcal{S}_{j'}}&=(d_{j}\cdots d_{j'-1}s_{j'-1}\iota\gamma,\flat(\flat(\cdots\flat(j',j'-1),j+1),j))\\
            &=(d_{j}\cdots d_{j'-2}\iota\gamma,j)\\
            &=\Theta^*\boxed{\gamma}_{\mcal{S}_{j}},
        \end{align*}
        as desired. 
        This concludes the construction of $\Box$. The triple 
        \(
            (\pi,\Box,\iota)\colon L\to M\times \EEx^{\Delta[1]}\times N
        \)
        factors through $(M\downarrow N)$ by \eqref{5WPV6OL} and \eqref{HDCZQA6}.
\end{proof}
    
\begin{remark}\label{DHSVOGS}
    Even though $\Box$ is the `constant exit loop inclusion,' the simplices $\boxed{\gamma}_{\mcal{S}_j}=(s_{j-1}(\gamma),j)$ are \emph{not} degenerate. Indeed, $(s_{j-1}(\gamma),j)=s_{j-1}(\gamma,e)$ requires $j=\sharp_{e,j-1}$, so $e\in\{j-1,j\}$, but $\sharp_{j-1,j-1}=j-1$ and $\sharp_{j,j-1}=j+1$.
\end{remark}

\begin{proposition}\label{T47D64N}
    Let $M\leftarrow L\to N$ be a linked space. Then the box map $\Box\colon L\hookrightarrow \down{M}{N}$ of \Cref{RPISW40} witnesses an equivalence $L\simeq\down{M}{N}$.
\end{proposition}

We start with the observation that \Cref{4S91CIG} holds with $\infty$-groupoids just as it does with topological spaces. We include a proof for completeness.

\begin{lemma}\label{lem:mapping cocylinder}
    Let $P_{\iota}$ be the mapping cocyclinder as in \Cref{dfn:paths starting in a subspace}. If $L$ and $N$ are $\infty$-groupoids, then $P_{\iota}\simeq L$.
\end{lemma}
\begin{proof}
    The source evaluation $N^{\Delta[1]}\rightarrow N^{\{0\}}$ is a Kan fibration between Kan complexes.
    Moreover, each fibre  $N^{\Delta[1]}_p\simeq p/N,$ $p\in N_0$, is contractible by virtue of being an under-$\infty$-groupoid. This verifies Condition (4) of \cite[00X2]{kerodon}, which implies that $N^{\Delta[1]}\rightarrow N^{\{0\}}$ is an equivalence, or, equivalently, a trivial Kan fibration. Kan fibrations are stable under pullback by \cite[00T5]{kerodon}, so the natural map $s\colon P_{\iota}\to L$ is a Kan fibration. Finally, as trivial Kan fibrations pull back to trivial Kan fibrations, $s$ is one such. As it is in particular a Kan fibration, \cite[00X2]{kerodon} implies that $s$ is an equivalence.
\end{proof}

The preceding lemma is a generalisation of the fact that under-$\infty$-groupoids are contractible, which is the special case when $L$ is a point. Since under-$\infty$-categories can be far from contractible, there is no reason to expect that \Cref{T47D64N} holds for linked $\infty$-categories. Indeed, most linked $\infty$-categories where ${N}$ contains a non-invertible morphism from $\iota({L})$ to ${N}$ provide counterexamples, e.g., $\{0\}\leftarrow\{0\}\hookrightarrow\Delta[1]\amalg_{\{0\}}\Delta[1]$. However, the following weaker statement holds for any linked $\infty$-category:

\begin{lemma}\label{J74MRSH}
    There is an isomorphism $\down{L}{N}\cong\down{M}{N}$.
\end{lemma}
\begin{proof}
    Observe $\down{M}{N}_{0}=\{\alpha\in\EEx_{1} : d_{1}\alpha\in M,\ d_{0}\alpha\in N\}=\Phat_{0}$ so that $\alpha\in \down{M}{N}_{0}$ iff $\alpha=(\Gamma,1)$ with $\Gamma\in (P_{L})_{0}=\down{L}{N}_{0}$. Thus, we have the map $\down{M}{N}_{0}\to\down{L}{N}_{0}$, $(\Gamma,1)\mapsto \Gamma$. This gives a bijection $\down{M}{N}_{0}\cong\down{L}{N}_{0}$.

    Generally, let $\alpha\colon \Delta[1]\times\Delta[n]\to\EEx $ be an element of $\down{M}{N}_{n}$, i.e., $\evv_{0}\alpha\in M_{n}$, $\evv_{1}\alpha\in N_{n}$. Then its restriction along any exit shuffle $\mcal{S}_{j}\colon\Delta[n+1]\hookrightarrow\Delta[1]\times\Delta[n]$ where $j\in\{1,\dots,n+1\}$ is vertical. Thus, 
    \(
        \alpha|_{\mcal{S}_{j}}=(\alpha_{j},j)\in\Phat_{n}
    \) 
    with 
    \(
        \alpha_{j}|_{0,\dots,j-1}\in \iota(L)_{j-1}.
    \)
    We will observe that the collection $\{\alpha_{j}\}\subset N_{j+1}$ assembles to give a map 
    \(
        A\colon\Delta[1]\times\Delta[n]\to N
    \)
    which lies within $\down{L}{N}_{n}$. Indeed, setting 
    \(
        A|_{\mcal{S}_{j}}\coloneqq\alpha_{j}
    \)
    defines $A$ on every non-degenerate $(n+1)$-simplex of $\Delta[1]\times\Delta[n]$ consistently since $\alpha$ itself is well-defined. 
    More precisely, let $\Theta\colon\Delta[\theta]\hookrightarrow\Delta[n+1]$ be some common simplicial subset of $\mcal{S}_{j}$ and $\mcal{S}_{j'}$ in $\Delta[1]\times\Delta[n]$. We have \(\Theta^* A|_{\mcal{S}_{j}}=\Theta^{*}A|_{\mcal{S}_{j'}}\) since $\alpha$ satisfies $\Theta^*\alpha|_{\mcal{S}_{j}}=\Theta^*\alpha|_{\mcal{S}_{j'}}$ so that in particular $\Theta^*\alpha_{j}=\Theta^*\alpha_{j'}$ in $N_{n+1-(j'-j)}$. This yields $A\in N^{\Delta[1]}_{n}=\down{N}{N}_{n}$. Since $\evv_{0}A|_{\mcal{S}_j}=\alpha_{j}|_{0,\dots,j-1}\in\iota(L)_{j-1}$ for any exit index $j$ as remarked above, we have $\evv_{0}A\in\iota(L)_{n}$, giving $A\in\down{L}{N}_{n}$. We have thus constructed a map 
    \(
        \Phi\colon \down{M}{N}\to\down{L}{N},\ 
        \alpha\mapsto  A.
    \)
    
    As for the inverse, let $\beta\colon\Delta[1]\times\Delta[n]\to N$ be an element of $\down{L}{N}_{n}$, i.e., $\evv_{0}\beta\in\iota(L)_{n}$, and let $j<j'$ be exit indices as above. Set 
    \(
        B|_{\mcal{S}_{j}}\coloneqq (\beta|_{\mcal{S}_{j}},j)\in\Phat_{n}.
    \)
    This is well-defined, since 
    \(
        \evv_{0}\mcal{C}_{j}\beta|_{\mcal{S}_{j}}=\beta|_{\mcal{S}_{j}}|_{0,\dots,j-1}=(\evv_{0}\beta)|_{0,\dots,j-1}\in\iota(L)_{j-1},
    \)
    so that $(\beta|_{\mcal{S}_{j}},j)$ is indeed an exit path of index $j$. We have $\Theta^*B|_{\mcal{S}_{j}}=\Theta^*B|_{\mcal{S}_{j'}}$ since $\Theta^*\beta|_{\mcal{S}_{j}}=\Theta^*\beta|_{\mcal{S}_{j'}}$. Thus, $\{B|_{\mcal{S}_{j}}\}$ assembles to give a map $B\colon\Delta[1]\times\Delta[n]\to \EEx$. Moreover, $B$ descends to $\down{M}{N}_{n}$, since the initial vertex of every $(\beta|_{\mcal{S}_{j}},j)$ -- or of any exit path -- is low, so that all vertices of $\evv_{0} B$ are low. We have thus constructed a map 
    \(
        \Psi\colon\down{L}{N}\to\down{M}{N},\ 
        \beta\mapsto B,
    \)
    the inverse to $\Phi$.
\end{proof}

\begin{proof}[Proof of \Cref{T47D64N}]
    We have $P_{\iota}=L\times_{N^{\{0\}}}N^{\Delta[1]}\cong\down{L}{N}$. The stated equivalence follows by composing \Cref{lem:mapping cocylinder} and \Cref{J74MRSH}. In order to show that it is realised by $\Box$, we will use notation from the proof of \Cref{J74MRSH}. Let $b\in L_{n}$, and let $\beta\colon\Delta[1]\times\Delta[n]\to N$ be the degenerate composition $\Delta[1]\times\Delta[n]\overset{\mrm{pr}}{\twoheadrightarrow}\Delta[n]\xrightarrow{\iota(b)}N$, so that $\beta\in\down{L}{N}$. For $j\in\{1,\dots,n+1\}$, we have $\boxed{b}_{\mcal{S}_{j}}=(s_{j-1}\iota(b),j)$, and
    \[
        \Psi(\beta)|_{\mcal{S}_{j}}=(\beta|_{\mcal{S}_{j}},j)=\paren{\paren{\Delta[n+1]\xhookrightarrow{\mcal{S}_{j}}\Delta[1]\times\Delta[n]\twoheadrightarrow\Delta[n]\xrightarrow{\iota(b)}N},j}.
    \]
    The underlying simplex map $\mrm{pr}\circ\mcal{S}_{j}\colon [n+1]\to[n]$ is $i\mapsto(0,i)\mapsto i$ if $i\leq j-1$ and $i\mapsto(1,i-1)\mapsto(i-1)$ if $i\geq j$, so $\mrm{pr}\circ\mcal{S}_{j}=s_{j-1}$. Thus 
    \(
        \Box=\Psi\circ\mrm{pr}^*
    \)
    where $\mrm{pr}^*\colon L\mapsto \down{L}{N}$, the ordinary constant loop inclusion, maps $b$ to $\beta$.
\end{proof}

\section{EPCs of CSSs}

\begin{theorem}\label{YSAL1EK}
    Let $X$ be a conically smooth stratified space of depth $1$ and $\mfrak{S}$ the corresponding linked space. Then $\EEx(\mfrak{S})\simeq \cat{E}{xit}(X)$.
\end{theorem}

In preparation for the proof, we need a few auxiliary constructions. They will be put to use to define an explicit comparison map $\EEx(\mfrak{S})\to\cat{E}{xit}(X)$; see \eqref{7PI7TP4}. Before we proceed, let us note the problem with the obvious Ansatz for such a map that these constructions will solve.
Let us assume for simplicity that $X$ has two strata so that $\mfrak{S}=(M\xleftarrow{\pi}L\xrightarrow{\iota}N)$.

Writing 
\[
\mrm{Unzip}=L\times\R_{\geq0}\amalg_{L\times\R_{>0}}N,
\]
recall from \cite{ayala2017local} that there is a pullback-pushout square 
\[
\begin{tikzcd}
    L\ar[r,hook]\ar[d,two heads,"{\pi}"] & \mrm{Unzip}\ar[d,hook]\\
    M\ar[r,hook] & X
\end{tikzcd}
\]
and an isomorphism 
\begin{equation}\label{A8U7AOO}
    X\cong|\mfrak{S}|\coloneqq C(\pi)\amalg_{L\times\R_{>0}}N
\end{equation}
which uses the \emph{fibrewise open cone}
\[
    C(\pi)=M\amalg_{L\times\{0\}}L\times\mbb{R}_{\geq0}
\]
and an open embedding $I\colon L\times\R_{>0}\hookrightarrow N$ satisfying $I_1=\iota$.

Now, given $(\gamma,e)\in\EEx_n$, we may see $\gamma$ within $|\mfrak{S}|$ by reparametrising $\mrm{Unzip}$ to $\mrm{Unzip}'=L\times\R_{\geq1}\amalg_{L\times\R_{>1}}N$ so that $\gamma$ lies within $\mrm{Unzip}$ and $\gamma|_{0,\dots,e-1}$ within the boundary $\mrm{Unzip}_0=L\times\{0\}$. However, in order for this to be a well-defined $n$-simplex of $X$, $\gamma$ outside of $\gamma|_{0,\dots,e-1}$ must lie wholly in the complement $N\smallsetminus L$ per condition \eqref{0PPIM43}, but this is not imposed on $(\gamma,e)$. We must therefore first elongate $\gamma|_{0,\dots,e-1}$ appropriately before embedding into $X$. The technical difficulty consists therefore in defining any functor $\EEx(\mfrak{S})\to\cat{E}{xit}(X)$ at all; that it is an equivalence will follow essentially from \Cref{T47D64N} and the corresponding result in the conically smooth context.

\begin{construction}\label{2J0Q1NN}
    The standard topological simplices may be alternatively constructed recursively, starting with the empty set, by taking closed cones. We will present a slightly different but equivalent version of \cite[Notation 2.3.2]{ayala2017stratified} and spell out how they assemble into a cosimplicial space $\Delta^\bullet\colon\mbf{\Delta}\to\mrm{Top}$. The details are provided for completeness.
    
    For a topological space $X$, set $D(X)=\ast$ if $X=\emptyset$ and 
    \[
        D(X)\coloneqq X\times[0,1]\amalg_{X\times\{1\}}\ast
    \]
    otherwise, defining the \emph{closed cone} on $X$. We may define $\Delta^n$ for $n\geq0$ iteratively by setting $\Delta^n=D(\Delta^{n-1})$ and $\Delta^{-1}=\emptyset$. We will write $n$ for the cone point, the equivalence class of $\ast\in D(\Delta^{n-1})$. 
    The face maps can be defined as follows: $\partial_n^n\colon \Delta^{n-1}\hookrightarrow\Delta^n$ is the inclusion of the base,
    \begin{equation}\label{SG83M4Z}
        \partial^n_n(x)=(x,0),
    \end{equation}
    and for $i<n$ we define $\partial_i^n\colon \Delta^{n-1}=D(\Delta^{n-2})\hookrightarrow D(\Delta^{n-1})=\Delta^n$ iteratively by 
    \(
        \partial^n_i(x,t)=(\partial_i^{n-1}(x),t).
    \)
    When $n=1$, we define $\partial_0^1\colon\Delta^0\hookrightarrow \Delta^1$ to be $\ast\mapsto (\ast,1)$. The degeneracy maps can be defined similarly: $\sigma^1_0\colon\Delta^1\to\Delta^0$ is unique, so let $n\geq2$. We define $\sigma^n_{n-1}\colon\Delta^n\to\Delta^{n-1}$ by 
    \[
        \sigma^n_{n-1}(x,t)=\begin{cases}
            x, & t<1\\
            n-1, & t=1.
        \end{cases}
    \]
    For $j<n-1$, we define $\sigma^n_j\colon\Delta^n\to\Delta^{n-1}$ by 
    \(
        \sigma^n_j(x,t)=(\sigma^{n-1}_j(x),t).
    \)
    Note that this implies $\sigma^n_{j}(n)=n-1$.
\end{construction}

\begin{construction}\label{O8ZQLHX}
    For $n\geq1$ and $e\in\{1,\dots,n\}$, let
    \[
        \Delta^n_e=\Delta^n\amalg_{\Delta^{e-1}\times\Delta^{\{1\}}}\Delta^{e-1}\times\Delta^1
    \]
    where the gluing map $\Delta^{e-1}\times\Delta^{\{1\}}\hookrightarrow\Delta^n$ is the standard inclusion. 
    The \emph{elongation map} is
            \begin{align*}
                E_n\colon\Delta^n&\to\Delta^n_n,\\
                (x,t)&\mapsto\begin{cases}
                    (x,2t), & t\leq \frac{1}{2}\\
                    (x,2t-1), & t\geq\frac{1}{2}
                \end{cases}
            \end{align*}
            where we wrote $(x,t)\in D(\Delta^{n-1})=\Delta^n$ according to \Cref{2J0Q1NN}. This is well-defined since at $t=\frac{1}{2}$ the gluing identifies $\Delta^{n-1}\times\Delta^{\{1\}}\ni(x,1)\sim(x,0)\in\Delta^n$ per \eqref{SG83M4Z}.
\end{construction}
\begin{construction}    
    For $n\geq1$ and $e\in\{1,\dots,n\}$, the \emph{collapse map} 
            \[
                K_e=K_e^n\colon\Delta^n_n\to\Delta^n_e
            \]
            is defined as follows: its restriction to $\Delta^n\subset\Delta^n_n$ is the identity, and its restriction to $\Delta^{n-1}\times\Delta^1$ is induced by the map 
            \[
                k_e\colon\Delta^{n-1}\times\Delta^1\to\Delta^{n-1}\times\Delta^1
            \]
            induced, using the fact that geometric realisation commutes with cartesian products, by the poset map, $[n-1]\times[1]\to[n-1]\times[1]$, 
            \begin{align*}
                (i,0)\mapsto\begin{cases}
                    (i,0), & i\leq e-1\\
                    (i,1), & i\geq e
                \end{cases}
            \end{align*}
            and $(i,1)\mapsto (i,1)$ for all $i$. We obtain a map to $\Delta^n_e$ by composing as follows:
            \[
                \Delta^{n-1}\times\Delta^1\xrightarrow{k_e}\Delta^{n-1}\amalg_{\Delta^{e-1}\times\Delta^{\{1\}}}\Delta^{e-1}\times\Delta^1\hookrightarrow \Delta^n_e
            \]
            where $k_e$ maps onto the second component and is well-defined on the union since $k_e|_{\Delta^{n-1}\times\Delta^{\{1\}}}=\id$, and the second map is given by $\partial_n\colon\Delta^{n-1}\hookrightarrow\Delta^n$ and $\id_{\Delta^{e-1}\times\Delta^1}$. 
\end{construction}

\begin{proof}[Proof of \Cref{YSAL1EK}]
    Without loss of generality, let $X$ and $\mfrak{S}$ be as discussed above \Cref{2J0Q1NN}. We will exhibit an equivalence 
    \[
        \mbf{R}\colon \EEx(\mfrak{S})\xrightarrow{\sim}\cat{E}{xit}(|\mfrak{S}|).
    \]
    It is the identity on objects, so let $n\geq1$. 
    We will use Lurie's quasi-category model for the target, where $\cat{E}{xit}_n(|\mfrak{S}|)$ consists of those paths $\gamma\colon\Delta^n\to|\mfrak{S}|$ for which there exists a sequence $p_0\leq\dots\leq p_n$ in $[1]$ satisfying 
    \begin{equation}\label{0PPIM43}
        s(\gamma(t_0,\dots,t_i,0,\dots,0))=p_i
    \end{equation}
    where $t_i\neq0$ and $s\colon|\mfrak{S}|\to[1]$ is the stratification (see \cite[\S A.6]{luriehigheralgebra}). 
    In terms of \Cref{2J0Q1NN}, condition \eqref{0PPIM43} reads $s(\gamma(x,t))=p_n$ for $t>0$ and $s(\partial^n_n(\gamma)(x',t'))=p_{n-1}$ for $t'>0$, and so on.
    
    An exit path $(\gamma,e)\in\EEx_n$ induces a map 
    \[
        \Delta^n_e\to|\mfrak{S}|
    \]
    as follows. Since $\mbf{b}\coloneqq\mbf{b}(\gamma,e)=\gamma|_{0,\dots,e-1}\in L_{e-1}$ (recall \Cref{3F1JHL6}), the map $\mbf{b}\times\id\colon \Delta^{e-1}\times\Delta^1\to L\times\R_{\geq0}$ is well-defined and factors through $L\times[0,1]$. Using $I_1=\iota$, this yields the desired map  
    \[
        \gamma\cup(\mbf{b}\times\id)\colon\Delta^n_e\to \mrm{Unzip}=L\times\R_{\geq0}\amalg_{L\times\R_{>0}}N\to|\mfrak{S}|
    \]
    where the final quotient projection identifies the $(e-1)$-path $\mbf{b}\times\{0\}$ with its projection in $M$. We now define $\mbf{R}(\gamma,e)$ to be the composition 
    \begin{equation}\label{7PI7TP4}
        \mbf{R}(\gamma,e)\colon\Delta^n\xrightarrow{E_n}\Delta^n_n\xrightarrow{K_e}\Delta^n_e\xrightarrow{\gamma\cup(\mbf{b}\times\id)}|\mfrak{S}|.
    \end{equation}
    It is a direct check that $\mbf{R}$ is functorial. 

    It now suffices to prove that the derivative 
    \(
        \mbf{R}_\ast\colon\down[\EEx]{M}{N}\to\down[\cat{E}{xit}]{M}{N}
    \)
    is an equivalence. To this end, consider the composition 
    \[
        \mbf{R}_\ast\circ\Box\colon L\to\down[\cat{E}{xit}]{M}{N}.
    \]
    For $\delta\in L_{n}$ and $j\in\{1,\dots,n+1\}$, we have
    \[
        \mbf{b}(s_{j-1}\delta,j)=(s_{j-1}\delta)|_{0,\dots,j-1}=\delta|_{0,\dots,j-1},
    \]
    so, recalling \eqref{K0R63TJ}, we obtain
    \[
        (\mbf{b}\boxed{\delta}_{\mcal{S}_j}\times\id)=\delta|_{0,\dots,j-1}\times\id\colon\Delta^{j-1}\times\Delta^1\to|\mfrak{S}|.
    \]
    Hence
    \[
        \mbf{R}\boxed{\delta}_{\mcal{S}_j}=(K_{j-1}E_n)^*(\boxed{\delta}_{\mcal{S}_j}\cup(\delta|_{0,\dots,j-1}\times\id)) \in\cat{E}{xit}_{n}(|\mfrak{S}|).
    \]
    Now, an equivalence $\Xi\colon\down[\cat{E}{xit}]{M}{N}\xrightarrow{\sim} L$ is given by sending a prism to its source face and restricting the second coordinate along $L\times\{0\}\hookrightarrow\mrm{Unzip}\to|\mfrak{S}|$. In order to obtain the source of the prism $\mbf{R}_\ast\boxed{\delta}$, we observe that if we take $j=n+1$, then 
    \[
        d_{n+1}(\boxed{\delta}_{\mcal{S}_{n+1}})=d_{n+1}(s_n\delta,n+1)=\pi(d_{n+1}s_n\delta)=\pi(\delta)
    \]
    and that $K_{n}=\id$, $E_n|_{d_n\Delta^n}=\id_{\Delta^{n-1}}$. Consequently, 
    \begin{equation}\label{CUC9INY}
        \Xi(\mbf{R}_\ast\boxed{\delta})=\delta,
    \end{equation}
    whence $\mbf{R}_\ast\circ\Box$ is an equivalence and so $\mbf{R}_\ast$ itself is also an equivalence by \Cref{RPISW40}. As $\mbf{R}$ is thus fully faithful, it is an equivalence.
\end{proof}

\section{Functoriality}\label{GPK27YP}

Ayala, Francis and Rozenblyum showed in \cite[Theorem 4.2.8]{ayala2017stratified} that the EPC construction embeds conically smooth stratified spaces fully faithfully into the $\infty$-category of $\infty$-categories (see \cite[\S 3]{lurie2009higher} or \cite[01YV]{kerodon}). In this section, we will prove the same for $\EEx$.

\subsection{The quasi-category of linked spaces}\label{PCO1GNW}

Let $\mrm{Poset}_{\leq 1}$ denote the ordinary category of posets of depth (at most) $1$, and $\mrm{Cat}$ the ordinary (and not the $2$-)category of categories. Readers familiar with d\'ecollages in the sense of \cite{barwick2020exodromy} can safely skip to \Cref{S2KANF2} and are referred to \Cref{VEJ5KDW}.

\begin{definition}
    The \emph{linking functor}
    \[
        \mbf{L}\colon \ocat{Poset}_{\leq1}\to\ocat{Cat}
    \]
    is defined as follows. Let $P\in\ocat{Poset}_{\leq1}$. For every object and every morphism of $P$, $\mbf{L}P$ has an object. We denote the object corresponding to $p\in P$ by $m_p\in\mbf{L}P$, and the object corresponding to $f\colon p_0\to p_1$ in $P$ by $l_f\in\mbf{L}P$. Besides the identities, $\mbf{L}P$ has two arrows as in 
    \[
        m_{p_0}\leftarrow l_f\to m_{p_1}
    \]
    for every morphism $f\colon p_0\to p_1$ in $P$. Note in particular that for every $p\in P$ there is a nontrivial span $m_{p}\leftarrow l_{\id_p}\to m_{p}$ at $m_p$, not merely the identity at $m_p$. There are no further objects or morphisms in $\mbf{L}P$.

    Given a morphism $F\colon P\to P'$ of posets, we define the functor $\mbf{L}F\colon\mbf{L}P\to\mbf{L}P'$ by setting $\mbf{L}F(m_p)=m_{F(p)}$, $\mbf{L}F(l_f)=l_{F(f)}$ on objects and \[
        \mbf{L}F(m_{p_0}\leftarrow l_f\to m_{p_1})=(m_{F(p_0)}\leftarrow l_{F(f)}\to m_{F(p_1)})
    \]
    on the nontrivial morphisms.
\end{definition}

\begin{remark} The linking functor is faithful but not fully so; e.g., the constant functors $(m_\ast\leftarrow l_{\id_\ast}\to m_\ast)=\mbf{L}\ast\to\mbf{L}\ast'$ given by $m_\ast,l_{\id_\ast}\mapsto m_{\ast'}$ and $m_\ast,l_{\id_\ast}\mapsto l_{\id_{\ast'}}$ are not in the image of $\mbf{L}|_{\Hom_{\mrm{Poset}_{\leq1}}(\ast,\ast')}$ which consists solely of the isomorphism $\mbf{L}\ast\cong\mbf{L}\ast'$
\end{remark}

\begin{definition}\label{8CDWKX5}
    The auxiliary category $\widehat{\mscc{LS}}=\widehat{\mscc{LS}}_{\leq1}\subset\ocat{Diag}(\cat{T}{op})$ of \emph{quasi-linked spaces} of depth (at most) $1$ is the subcategory of the category of topological diagrams whose shapes are linking posets, i.e., 
    \[
        \widehat{\mscc{LS}}_{\leq1}=\ocat{Poset}_{\leq1}\times_{\ocat{Cat}}\ocat{Diag}(\cat{T}{op}).
    \]
    where $\ocat{Diag}(\cat{T}{op})\to\ocat{Cat}$, $(D\to\cat{T}{op})\mapsto D$, maps a diagram to its shape.
    It is naturally a topological category by setting 
    \begin{align*}
        \Hom_{\widehat{\mscc{LS}}}((P,\mfrak{S}),(Q,\mfrak{T}))=
        \coprod_{\Phi\colon P\to Q}\prod_{f\colon p_0\to p_1} \Big([\mfrak{S}_{p_0},\mfrak{T}_{\mbf{L}\Phi(p_0)}]\times_{[\mfrak{S}_{f},\mfrak{T}_{\mbf{L}\Phi(p_0)}]}\\ [\mfrak{S}_{f},\mfrak{T}_{\mbf{L}\Phi(f)}]\times_{[\mfrak{S}_f,\mfrak{T}_{\mbf{L}\Phi(p_1)}]} [\mfrak{S}_{p_1},\mfrak{T}_{\mbf{L}\Phi(p_1)}]\Big)
    \end{align*}
    where the disjoint union is over poset maps, the product is over arrows in $P$, and $[-,-]=\Hom_{\cat{T}{op}}(-,-)$. 
\end{definition}

Concretely, a quasi-linked space is a pair
\(
    (P,\mfrak{S}\colon\mbf{L}P\to\cat{T}{op})
\) 
where $P\in\ocat{Poset}_{\leq1}$ and $\mfrak{S}$ is a functor. Given $f\colon p_0\to p_1$ in $P$, we write
\(
    \mfrak{S}(m_{p_0}\leftarrow l_{f}\to m_{p_1})=(\mfrak{S}_{p_0}\xleftarrow{\pi_{f}}\mfrak{S}_{f}\xrightarrow{\iota_{f}} \mfrak{S}_{p_1})=(M_{p_0}\xleftarrow{\pi_{f}}L_f\xrightarrow{\iota_{f}} M_{p_1}).
\)
A morphism $\phi\colon(P,\mfrak{S}\colon\mbf{L}P\to\cat{T}{op})\to(Q,\mfrak{T}\colon\mbf{L}Q\to\cat{T}{op})$ is a pair 
\(
    (\Phi\colon P\to Q,\phi\colon\mfrak{S}\to\mfrak{T})
\)
where $\Phi$ is a morphism in $\ocat{Poset}_{\leq1}$ and $\phi$ is a natural transformation 
\(
    \phi\colon\mfrak{S}\Rightarrow\mfrak{T}\circ\mbf{L}\Phi.
\)
We will sometimes suppress the posets in notation.

\begin{remark}\label{Y58GWST}
    Due to the parametrisation over poset maps mediated by the linking functor, $\widehat{\mscc{LS}}$ differs substantially from the category of topological spans. For instance, a $P$-stratified quasi-linked space consists, for $P=[1]$, of five spaces organised into three spans: 
    \[
        \mfrak{S}=(L_{\id_0}\rightrightarrows M_{0}\leftarrow L_{0<1}\to M_1\leftleftarrows L_{\id_1}).
    \]
    A map $\phi\colon([1],\mfrak{S})\to([1],\mfrak{S}')$ need not map the spans in $\mfrak{S}$ to the corresponding spans in $\mfrak{S}'$; rather, when the accompanying poset map $\Phi\colon[1]\to[1]$ factors, say, through $\{0\}$, then $\phi\colon\mfrak{S}\Rightarrow\mfrak{S}'\circ\mbf{L}\Phi$ consists of three topological span maps 
    \[
        (L_{\id_0}\rightrightarrows M_0),(M_0\leftarrow L_{0<1}\to M_1), (M_1\leftleftarrows L_{\id_1})\to (L'_{\id_0}\rightrightarrows M'_0)
    \]
    such that the maps out of $M_0$ and $M_1$ coincide in all three but with no further conditions.
\end{remark} 

\begin{definition}\label{S2KANF2}
    The category $\mscc{LS}\subset\widehat{\mscc{LS}}$ of \emph{linked spaces} of depth (at most) $1$ is the full topological subcategory of $\widehat{\mscc{LS}}$ consisting of those quasi-linked spaces $(P,\mfrak{S}\colon\mbf{L}P\to\cat{T}{op})$ which satisfy the following conditions:
    \begin{enumerate}
        \item For every non-identity morphism $f$ in $P$, the map $\pi_f$ is a fibration and the map $\iota_f$ is a cofibration.
        \item For every object $p$ in $P$, we have $L_{\id_p}=M_p$ and $\pi_{\id_p}=\iota_{\id_p}=\id_{M_p}$. 
    \end{enumerate}
\end{definition}

\begin{remark}\label{ZCOQH4R}
    \Cref{S2KANF2} removes the extravagance of \Cref{8CDWKX5} pointed out in \Cref{Y58GWST}. A map $\phi\colon([1],\mfrak{S})\to([0],\mfrak{S}')$ as discussed there is, in $\mscc{LS}$, determined by a single span map $(M_0\leftarrow L_{0<1}\to M_1)\to (M_0'\rightrightarrows M_0')$. Such a span map is tantamount to a map 
    \(
        M_0\amalg_{L_{0<1}}M_1\to M_0'
    \)
    so in particular every conically smooth map from a depth-$1$ space into a smooth manifold is of this type.
    This is analogous to the condition that the lower stratum in an ordinary stratified space be in the closure of the higher.
\end{remark}
\begin{remark}\label{VEJ5KDW}
    The secondition condition of \Cref{S2KANF2} makes $\mscc{LS}$ a subcategory of $\mrm{Diag}_P=\Fun(R(P)^{\mrm{op}},\mrm{Top})$, the topological version of the category of diagrams over $P$ in the sense of Douteau: see \cite[Definition 1.10]{douteau2021homotopy}. Over depth-$1$ posets, these coincide with spatial d\'ecollages in the sense of Barwick--Glasman--Haine \cite[\S 2.6]{barwick2020exodromy}. Removing said condition amounts to allowing degenerate subdivisions of $P$. We have chosen to circumvent the construction of $R$ for concreteness as we only consider depth-$1$ posets.
\end{remark}

\begin{example}\label{UTL5VOM}
    Consider the linked space $\mrm{sp}=([1],\mrm{sp})$ where $\mrm{sp}=(\ast\leftarrow\ast\to\ast)$. A map $\mrm{sp}\to M=([0],\mrm{const}_M)$ is determined by a point in $M$. In contrast, a map $\phi\colon\mrm{sp}\to([1],\mfrak{S})$ with $\Phi=\id_{[1]}$ is determined by a link point $l\in L=\mfrak{S}_{0<1}$, which can be interpreted as an `infinitesimal exit path' $\pi(l)\to\iota(l)$ via the box map of \Cref{RPISW40}.
\end{example}

\begin{example}
    Consider the poset $I=\{\pm<_{\pm}1\}$ and the linked space $\mfrak{I}=(I,\mfrak{I}\colon\mbf{L}I\to\cat{T}{op})$ where $\mfrak{I}_-=\ast=\mfrak{I}_{+}$, $\mfrak{I}_1=\R$, $\mfrak{I}_{<_{\pm}}=\ast$, and $\iota_{<_-}=0$, $\iota_{<_+}=1$. Let now $\phi\colon\mfrak{I}\to([1],\mfrak{S})$ be a map satisfying $\Phi(0)=0$ and $\Phi(+)=\Phi(1)=1$. Then $\phi\colon \mfrak{I}\Rightarrow\mfrak{S}\circ\mbf{L}\Phi$ is a diagram of type 
    \[
    \begin{tikzcd}[sep=small]
        & \ast\ar[dl]\ar[dr,"{0}"]\ar[dd] & & \ast\ar[dl,"{1}"']\ar[dr]\ar[dd]\\
        \ast\ar[dd] & & \R\ar[dd,"{\gamma}"] & &\ast\ar[dd] \\
        & L\ar[dl,two heads,"{\pi}"]\ar[dr,hook,"{\iota}"'] & & N\ar[dl,"{\id}"]\ar[dr,"{\id}"'] & \\
        M & & N & & N
    \end{tikzcd}
    \]
    so $\gamma$ specifies an exit path in $\EEx(\mfrak{S})$ from $\pi(\gamma(0))$ to $\gamma(1)$, namely $(\gamma|_{[0,1]},1)\in\Phat_0$. Conversely, any exit path may be extended to a map $\R\to N$ constantly at its endpoints to define a map $\phi\colon(I,\mfrak{I})\to([1],\mfrak{S})$ with $\Phi$ as above. 
\end{example}

\begin{definition}
    The \emph{$\infty$-category of linked spaces of depth at most $1$} is the homotopy coherent nerve $\Nhc(\mscc{LS})$ of the topological category $\mscc{LS}=\mscc{LS}_{\leq1}$ of linked spaces of depth at most $1$. We will denote it by $\mscc{LS}$ if there is no risk of confusion.
\end{definition}

\subsection{Fully-faithfulness}\label{8592U1F}

Let $\Ctop$ denote the locally-Kan simplicial category of all $\infty$-categories, where $\Hom_{\Ctop}(\mscc{C},\mscc{D})=\mrm{Fun}^{\simeq}(\mscc{C},\mscc{D})$ is the core of the functor $\infty$-category $\mrm{Fun}_\bullet(\cat{C},\cat{D})=\Hom_{\mrm{sSet}}(\cat{C}\times\Delta^\bullet,\cat{D})$ generated by those natural transformations which are natural isomorphisms. As in \cite[\S 3]{lurie2009higher}, we will denote by $\Cinfty=\Nhc(\Ctop)$ the homotopy-coherent nerve of $\Ctop$.

\begin{construction}\label{5X56IK6}
We will promote the EPC construction to a topological functor
\[
    \EEx\colon \mscc{LS}\to\Ctop.
\]
It will suffice to consider maps between posets of at most two elements, so let $([1],\mfrak{S}), ([1],\mfrak{T}), ([0],\mfrak{U})\in\mscc{LS}$. 
A map $\phi\colon([1],\mfrak{S})\to([1],\mfrak{T})$ with $\Phi=\id_{[1]}$ is an ordinary span map (cf.\ \Cref{ZCOQH4R}) and naturally induces an $\infty$-functor 
\[
    \phi_!\colon\EEx(\mfrak{S})\to\EEx(\mfrak{T}).
\] 
The case $\Phi\neq\id_{[1]}$ is tantamount to a map $\phi\colon([1],\mfrak{S})\to([0],\mfrak{U})$. We set 
\[
    \EEx([0],\mfrak{U})=\mfrak{U}_0,
\]
suppressing the distinction between a space and its complex of singular chains. Now, $\phi\colon\mfrak{S}\Rightarrow\mfrak{U}\circ\mbf{L}\Phi$ is a span map of type $(M\leftarrow L\to N)\to(M'\rightrightarrows M')$ which is already determined by the map $\phi_1\colon N\to M'$ (cf.\ \Cref{ZCOQH4R}). This induces an $\infty$-functor 
\[
    \phi_!\colon\EEx(\mfrak{S})\to M'
\]
by mapping an exit path $(\gamma,e)\in\Phat_{k-1}\subset\EEx_k$ to 
\[
    \phi_!(\gamma,e)=\phi_1(\gamma).
\]
It is easily seen that this is functorial since $\phi_0(\pi(\mbf{b}(\gamma,e)))=\phi_1(\mbf{b}(\gamma,e))$.
Finally, a map of type $\phi\colon([0],\mfrak{U})\to([1],\mfrak{S})$ is tantamount to an ordinary map $M'\to M$ or $M'\to N$ depending on $\Phi\colon[0]\to[1]$.

This defines $\EEx$ as an ordinary functor. We will now define its extension
\[ 
    \Hom_{\mscc{LS}}(([1],\mfrak{S}),([1],\mfrak{T}))_\bullet\to\Hom_{\Ctop}(\EEx(\mfrak{S}),\EEx(\mfrak{T}))_\bullet
\]
onto the full morphism spaces. 
Recalling \Cref{8CDWKX5}, let us write $\Hom_{\mscc{LS}}=\mcal{F}_{0}\amalg\mcal{F}_1\amalg\mcal{F}_{01}$ where $\mcal{F}_{i}$ is the connected component at $\Phi=\sigma_{i}\colon [1]\to[1]$ for $i=0,1$ and $\mcal{F}_{01}$ is the connected component at $\Phi=\id_{[1]}$.
If $L=\mfrak{S}_{0<1}=\emptyset$, we set $(\phi_i)_!|_{X}=(\phi_i)|_X$ for $i=0,1$ and $X=M,N$.

Suppose now that $L\neq\emptyset$. By construction, an $n$-simplex in $\mcal{F}_0$ is of type 
\[
    H=(H_M,H_L,H_N)\in[M,M']_n\times_{[L,M']_n}[L,M']_n\times_{[L,M']_n}[N,M']_n.
\]
where $H_-\colon -\times\Delta[n]\to M$. We define the restriction to the core by $H_!|_{\EEx(\mfrak{S})^{\simeq}}=H_M\amalg H_N\colon\EEx(\mfrak{S})^{\simeq}=M\amalg N\to M'\subset\EEx(\mfrak{T})$. Now let $(\gamma,e)\in\Phat_{k-1}\subset\EEx_k(\mfrak{S})$ and let $\Gamma\colon\Delta[k]\to\EEx(\mfrak{S})\times\Delta[n]$ be given such that $\Gamma_1=(\gamma,e)$. Then $\Gamma$ induces 
\[
    \widetilde{\Gamma}=(\gamma,\Gamma_2)\colon\Delta[k]\to N\times\Delta[n]
\]
and thus $H_N(\widetilde{\Gamma})=H_N\circ\widetilde{\Gamma}\colon\Delta^k\to M'$. We set 
\begin{equation}\label{WDGZLBT}
    H_!(\Gamma)=H_N(\widetilde{\Gamma}).
\end{equation}
This gives a map $H_!\colon\EEx(\mfrak{S})\times\Delta[n]\to M'$ which defines 
\(
    \Hom_{\mscc{LS}}\supset\mcal{F}_0\to\mscc{F}_0\subset\Hom_{\Ctop}
\),
$H\mapsto H_!$.
The restriction $\mcal{F}_1\to\mscc{F}_1$ is induced analogously by replacing $M'$ with $N'=\mfrak{T}_1$ throughout.
Similarly, given an $n$-simplex 
\(
    H=(H_M,H_L,H_N)\in[M,M']_n\times_{[L,M']_n}[L,L']_n\times_{[L,N']_n}[N,N']_n
\) 
in $\mcal{F}_{01}$, where $L=\mfrak{T}_{0<1}$, and given $\Gamma=((\gamma,e),\Gamma_2)\in(\EEx(\mfrak{S})\times\Delta^n)_k$, $\widetilde{\Gamma}=(\gamma,\Gamma_2)\in (N\times\Delta[n])_k$ as above, we set 
\begin{equation}\label{4F73MJK}
    H_!(\Gamma)=(H_N(\widetilde{\Gamma}),e)
\end{equation}
which defines the restriction $\mcal{F}_{01}\to\mscc{F}_{01}$.
\end{construction}

\begin{theorem}\label{JG5DRK5}
    The topological EPC functor of \Cref{5X56IK6} induces a fully faithful functor
    \[
        \EEx\colon\mscc{LS}_{\leq1}\hookrightarrow\Cinfty.
    \]
    of $\infty$-categories.
\end{theorem}

The reason for our passing to homotopy-coherent nerves is that $\EEx=(-)_!\colon \Hom_{\mscc{LS}}\to\Hom_{\Ctop}$ is merely a weak equivalence, as we will show. Since the morphism spaces in homotopy-coherent nerves are equivalent to those in the original locally-Kan category (this is due to Joyal; see also Hebestreit--Krause \cite{hebestreit2020mapping} for a direct proof), \Cref{JG5DRK5} will follow.

Before we prove \Cref{JG5DRK5}, we will use it to produce three classes of counterexamples to the three weak versions of the conically smooth stratified homotopy hypothesis discussed in \Cref{H09VB9M}. 
\subsubsection{Class I: non-finite local links}

Let $B$ be a connected smooth manifold without boundary whose fundamental group is not purely torsion so that the fibres of its universal cover $\pi\colon \widetilde{B}\to B$ have infinitely many connected components. Let $\widetilde{B}\to Y$ be any cofibration. Each can be chosen finite, e.g., $(S^1\leftarrow \mbb{R}\hookrightarrow\mbb{R}^2)$. The idea is to violate the compactness of the local links in a CSS. 

\begin{corollary}\label{JWCBC9T}
    The $[1]$-layered $\infty$-category $\EEx(B\leftarrow \widetilde{B}\to Y)$ is not equivalent to the EPC of a CSS. If $B, \widetilde{B}$ and $Y$ are finite, then so is $\EEx(B\leftarrow \widetilde{B}\to Y)$.
\end{corollary}
\begin{proof}
    Suppose $X$ is a CSS with associated linked manifold $(M\xleftarrow{\pi'} L\to N)$ such that $\cat{E}{xit}(X)\simeq\EEx(B\leftarrow \widetilde{B}\to Y)$. By \Cref{YSAL1EK} we have $\EEx(M\leftarrow L\to N)\simeq\EEx(B\leftarrow \widetilde{B}\to Y)$ and by \Cref{JG5DRK5} we have a homotopy equivalence $(B\leftarrow \widetilde{B}\to Y)\simeq (M\leftarrow L\to N)$ of spans. In particular, we obtain the commutative square
    \[
    \begin{tikzcd}
        \widetilde{B}\ar[d,"{\pi}"]\ar[r,"{\sim}"',"{\overline{f}}"] & L\ar[d,"{\pi'}"]\\
        B\ar[r,"{\sim}"',"{f}"] & M
    \end{tikzcd}
    \]
    where the vertical maps are fibre bundles and $\overline{f}$ is a homotopy equivalence covering the homotopy equivalence $f$, i.e., a fibre-homotopy equivalence. Therefore, the fibres of $\pi$ and $\pi'$ are homotopy equivalent. The fibres of $\pi'$ are compact since around each point $q\in M$ there exists a basic open of type $U\cong \mbb{R}^k\times C(Z)$ such that $(\pi')^{-1}(q)\cong Z$ where $Z$ is a closed smooth manifold. But the fibres of $\pi$ have infinitely many connected components by assumption, contradicting the compactness of $Z$.
    The second statement follows from \cite[Proposition 2.11]{volpe2024finitenessfinitedominationstratified}.
\end{proof}

\begin{remark}
    The local links of $\EEx(B\leftarrow\widetilde{B}\to Y)$ are not finite. A similar example was given by Volpe in \cite[Remark 2.14]{volpe2024finitenessfinitedominationstratified}.  
\end{remark}

\subsubsection{Class II: finite local links}

Besides the compactness of the spaces $Z$ that appear in the proof of \Cref{JWCBC9T}, another characteristic property of a CSS over $[1]$ is that the link embedding $\iota\colon  L\hookrightarrow N$ has a (rank-$1$) normal framing. This motivates a second class of examples.

Recall that a smooth $n$-manifold $Y$ is called \emph{stably (normally) frameable} if there exists a natural number $k\geq0$ together with a bundle isomorphism $\Tangent Y\oplus\varepsilon^k\cong\varepsilon^{n+k}$, where $\varepsilon$ denotes the trivial real bundle of rank $1$. A closed embedding $Y\hookrightarrow\mbb{R}^{K}$ whose normal bundle is trivialisable makes $Y$ stably frameable. The Stiefel--Whitney classes of a stably frameable smooth manifold vanish. Contractible smooth manifolds are stably frameable, and if there is a homotopy equivalence $Y\simeq Y'$ of compact manifolds and the Stiefel--Whitney classes of $Y$ vanish, then so do those of $Y'$ (using Wu's formula and the homotopy invariance of Steenrod operations).

\begin{corollary}\label{DM62IBZ}
    Let $\Lambda$ be compact smooth manifold with a non-vanishing Stiefel--Whitney class, and let $\Lambda\hookrightarrow \mbb{R}^K$ be a closed embedding. Then the finite $[1]$-layered $\infty$-category $\EEx(\ast\leftarrow \Lambda\xhookrightarrow{} \mbb{R}^K)$ has contractible strata, finite local links, and is not equivalent to the EPC of a CSS.
    \end{corollary}
    \begin{proof}
    Suppose that $X$ is a CSS with associated linked manifold $(M\leftarrow L\to N)$ such that $\cat{E}{xit}(X)\simeq\EEx(\ast\leftarrow \Lambda\to \mbb{R}^K)$. By \Cref{YSAL1EK} we have $\EEx(M\leftarrow L\to N)\simeq\EEx(\ast\leftarrow \Lambda\to\mbb{R}^K)$ and by \Cref{JG5DRK5} we have a homotopy equivalence $(M\leftarrow L\to N)\simeq (\ast\leftarrow \Lambda\to \mbb{R}^K)$ of spans. In particular, $N$ is contractible, and since $X$ is a depth-$1$ CSS, the map $L\xhookrightarrow{} N$ has a normal framing. Hence $L$ stably frameable. Moreover, the map $L\to M$ is a smooth fibre bundle with closed smooth fibre $Z$ over the contractible smooth manifold $M$, making $Z$ also stably frameable.\footnote{Let $\Tangent L\oplus \varepsilon^k\cong\varepsilon^N$ be a stable framing on $L$. Over a contractible trivialising neighbourhood $U\subset M$ we have a bundle isomorphism $\Tangent(Z\times U)\cong \Tangent Z \oplus \varepsilon^{\dim M}$. Since $Z\times U$ is open in $L$, we have $\varepsilon^{N}\cong\Tangent L|_{Z\times U}\oplus \varepsilon^k\cong\Tangent (Z\times U)\oplus \varepsilon^k\cong\Tangent Z\oplus \varepsilon^{\dim M}\oplus \varepsilon^k$, yielding a stable framing on $Z$.} But then $Z\simeq Z\times M\simeq L\simeq\Lambda$ -- a contradiction.
    \end{proof}

\subsubsection{Class III: contractible local links}\label{PA27WQL} 

Recall that a space is called a \emph{homology sphere} if its singular homology coincides with that of a sphere. 
Lefschetz duality implies that the boundary of a compact contractible manifold is a homology sphere. This leads to the following:

\begin{corollary}\label{761A5V7}
    Let $Y$ be a closed smooth manifold which is not a homology sphere and let $Y\hookrightarrow\mbb{R}^K$ be a closed embedding. Then the finite $[1]$-layered category $\EEx(Y\xleftarrow{\id}Y\hookrightarrow\mbb{R}^K)$ has contractible local links and is not equivalent to the EPC of a compact CSS.
\end{corollary}
\begin{proof}
    For simplicity, let us assume $Y$ connected. Suppose $X$ is a compact CSS with associated linked space $(M\leftarrow L\to N)$. Using \Cref{YSAL1EK,JG5DRK5} we obtain a homotopy equivalence $(Y\leftarrow Y\to\mbb{R}^K)\simeq(M\leftarrow L\to N)$ of spans. In particular, $M$ is connected and $N$ is contractible. Let now $x\in M$ be given and let $U\cong\mbb{R}^{n-1}\times C(Z)$ be a basic open centred at $x$. Then $L\to M$ is a smooth $Z$-bundle homotopy equivalent to $Y$ over a base that is homotopy equivalent to $Y$, so $Z$ is contractible. Since $Z$ is a closed smooth manifold, we obtain $Z=\ast$. But then $U\cong\mbb{H}^n$ is the $n$-dimensional upper half-space, and the conically smooth atlas on $X$ gives it the structure of a compact smooth manifold with boundary such that $M=\partial X$ and $N=X^{\circ}$. But then $M$ is a homology sphere -- a contradiction.
\end{proof}

\subsubsection{The proof of \Cref{JG5DRK5}}

Let us observe that the existence of different classes of maps of linked spaces depending on the accompanying poset maps is also reflected in maps between EPCs. Throughout, let $([1],\mfrak{S}), ([1],\mfrak{T})\in\mscc{LS}$.

\begin{notation*}
    We write $\mscc{F}_0 \subset\Hom_{\Ctop}(\EEx(\mfrak{S}),\EEx(\mfrak{T}))$ for the subspace consisting of those functors $F\colon\EEx(\mfrak{S})\to\EEx(\mfrak{T})$ which factor through $M'=\mfrak{T}_0$; $\mscc{F}_1$ for those which factor through $N'=\mfrak{T}_1$, and $\mscc{F}_{01}=\Hom_{\Cinfty}(\EEx(\mfrak{S},\EEx(\mfrak{T})))\smallsetminus(\mscc{F}_0\cup\mscc{F}_1)$.
\end{notation*}

\begin{lemma}
    $\Hom_{\Ctop}(\EEx(\mfrak{S}),\EEx(\mfrak{T}))=\mscc{F}_0\amalg\mscc{F}_1\amalg\mscc{F}_{01}$.
\end{lemma}
\begin{proof}
    The paths in question are natural isomorphisms.
\end{proof}

\begin{definition}
    Let $\mfrak{S}$ and $\mfrak{T}$ be two linked spaces. We call the image of $\EEx\colon\Hom_{\mscc{LS}}(\mfrak{S},\mfrak{T})\to\Hom_{\Ctop}(\EEx(\mfrak{S}),\EEx(\mfrak{T}))$ the space of \emph{tame maps} (and higher tame homotopies), and will denote it by $\Hom^\tame(\EEx(\mfrak{S}),\EEx(\mfrak{T}))$.
\end{definition}

\begin{remark}\label{3REAG6N}
    Let $F\colon\EEx(\mfrak{S})\to\EEx(\mfrak{T})$ be given and suppose, without loss of generality, that $F\in\mscc{F}_{01}$. The induced map $F_*\colon\EEx(\mfrak{S})^{\Delta[1]}\to\EEx(\mfrak{T})^{\Delta[1]}$, which we will call the \emph{derivative} of $F$, restricts to the map 
    \[
        F_*\colon\down{M}{N}\to\down{M'}{N'}.
    \]
    Note, however, that $F$ induces another, similar map:
    \[
        (F|_N)_*'\colon\down{M}{N}\xrightarrow{\Phi}\down{L}{N}\xrightarrow{(F|_N)_*}(N')^{\Delta^1}
    \]
    where $\Phi$ is the isomorphism from \Cref{J74MRSH}, with inverse $\Psi$. If $F$ is tame, then (recall \eqref{WDGZLBT} and \eqref{4F73MJK}) $(F|_N)_*$ factors through $\down{L'}{N'}\hookrightarrow(N')^{\Delta^1}$ and $F_*$ is in fact equal to the resulting map 
    \[
    \down{M}{N}\xrightarrow{\Phi}\down{L}{N}\xrightarrow{(F|_N)_*}\down{L'}{N'}\xrightarrow{\Psi}\down{M'}{N'}.
    \]
    If $F$ is not tame, this need not be true, as the following example illustrates in the simplest case.
\end{remark}

\begin{example}
    Recall from \Cref{UTL5VOM} the linked space $\mrm{sp}=(\ast\leftarrow\ast\to\ast)$ and that a map $\mrm{sp}\to\mfrak{S}$ is determined by a single link point $l\in L$, which is in turn determined by the exit path $\boxed{l}=(s(l),1)$ with degenerate underlying path in $N$. Conversely, all tame maps are of this type, but not every map $F\colon\EEx(\mrm{sp})=\ast^{\triangleleft}\to\EEx(\mfrak{S})$ is of this type; instead, $F$ is determined by an arbitrary exit path $F(s(\ast,1))=(\gamma,1)$ where $\gamma$ need not be constant. Nevertheless, $F$ is homotopic to the tame map $F^t=(\pi(\gamma(0)),\gamma(0),\gamma(0))_!\colon\EEx(\mrm{sp})\to\EEx(\mfrak{S})$. 
\end{example}

In order to prove \Cref{JG5DRK5}, we will show that the map on morphism spaces is an equivalence by showing that it in turn is essentially surjective and fully faithful. We start with the former:

\begin{lemma}\label{4TH5BVK}
    Every map $F\colon\EEx(\mfrak{S})\to\EEx(\mfrak{T})$ is homotopic to a tame map.
\end{lemma}

\begin{proof}
    Suppose $F\in\mscc{F}_{01}$; the cases $F\in\msc{F}_0$ and $F\in\msc{F}_1$ are analogous. The induced map 
    \[
        h^F\colon L\xhookrightarrow{\Box} \down{M}{N}\xrightarrow{F_*}\down{M'}{N'}\overset{\Phi}{\cong}\down{L'}{N'},
    \]
    fits into the diagram 
    \[
    \begin{tikzcd}
        L \ar[d,"{\iota}"',hook]\ar[r,"{h^F}"] & (N')^{\Delta^1}\ar[d,"{\evv_1}"]\\
        N\ar[r,"{F|_{N}}"']\ar[ur,dotted,"{H^F}"] & N'
    \end{tikzcd}
    \]
    where the lift $H^F$ exists because $\iota$ is a cofibration. Let us write $h^F_0\colon L\to L'$ for the evaluation at $0$, and similarly for $H^F_0\colon N\to N'$. Now, the triple
    \[
        F^\tame=(F|_M,h^F_0,H^F_0)\colon\mfrak{S}\to\mfrak{T}
    \] 
    is a span map, and we claim that there is a homotopy
    \[
        F\sim F^\tame_!
    \]
    in $\Hom_{\Ctop}(\EEx(\mfrak{S}),\EEx(\mfrak{T}))$.

    On $M$, we take the homotopy to be constant, and on $N$ we take it to be given by $H$ itself. It remains to extend it to exit paths. First, observe that such an extension can be extended from a homotopy of the following type. There are two induced derivatives 
    \[
        F'_\ast, (H^F_0)_\ast\colon \down{L}{N}\to\down{L'}{N'}
    \]
    where we write 
    \[
        F_*'= \Phi'F_*\Psi\colon \down{L}{N}\xrightarrow{\Psi}\down{M}{N}\xrightarrow{F_*}\down{M'}{N'}\xrightarrow{\Phi'}\down{L'}{N'}.
    \] 
    Note that $(\gamma,e)=\mcal{S}_{e}^{\ast}\mcal{C}_{e}^{\ast}(\gamma,e)=\mcal{S}_e^*\Psi\Phi\mcal{C}_e^*(\gamma,e)$, thus
    \begin{align*}
        F(\gamma,e)&=\mcal{S}_e^*F_*\mcal{C}_e^*(\gamma,e)\\
        &=\mcal{S}_e^*\Psi'F'_*\Phi\mcal{C}_e^*(\gamma,e)
    \end{align*}
    and similarly for $H^F_0$. 
    Consequently, if there is a homotopy 
    \[
        (\mbf{H}\colon F'_\ast\to(H^F_0)_\ast)\colon\down{L}{N}\to {\down{L'}{N'}}^{\Delta[1]}
    \] 
    which it covers the homotopy $(\id_{h^F_0},H^F)\colon (L\times N)\to \Map(\Delta^1, L'\times N')$ along the evaluations at $0$ and $1$, it would induce, with a slight abuse of notation, the full homotopy on exit paths:
    \begin{align*}
        \mbf{H}\colon\EEx(\mfrak{S})&\to\Fun(\Delta[1],\EEx(\mfrak{T}))\\
        (\gamma,e)&\mapsto(\mbf{H}\colon\mcal{S}_e^*\Psi'F'_*\Phi\mcal{C}_e^*(\gamma,e)\to \mcal{S}_e^*\Psi'(H^F_0)_*\Phi\mcal{C}_e^*(\gamma,e)).
    \end{align*}
Here, the latter $\mbf{H}$ stands more precisely for $\Psi'$ applied to the map 
\begin{equation}\label{NCOTBV1}
    \Delta[k]\times\Delta[1]\xrightarrow{\mcal{S}_{e}\times\id}\Delta[1]\times\Delta[k-1]\times\Delta[1]\xrightarrow{\mbf{H}(\Phi\mcal{C}_{e}^*(\gamma,e))}N'
\end{equation}
where we used $\mbf{H}(\Phi\mcal{C}_e^*(\gamma,e))\in({\down{L'}{N'}}^{\Delta[1]})_{k-1}\subset\Hom(\Delta[1]\times\Delta[k-1],\Fun(\Delta[1],N'))\cong\Hom(\Delta[1]\times\Delta[k-1]\times\Delta[1],N')$.
The condition that $\mbf{H}$ cover $(\id_{h^F_0},H^F)$ ensures that the above, together with its restrictions to $M$ and $N$ as already specified, is functorial. Since $(F^\tame)'_*=(H^F_0)_*$ holds by construction (recall \eqref{WDGZLBT} and \eqref{4F73MJK}), this would be a homotopy of the required type. 

It remains to construct the desired homotopy $\widehat{H}\colon F'_*\to(H^F_0)_*$.\footnote{Of course $F'_*$ and $(H^F_0)'_\ast$ are homotopic since their restrictions along the equivalence $\Phi\Box\colon L\hookrightarrow \down{L}{N}$ coincide, but a priori not over $(\id_{h^F_0},H^F)$. This will be evident in our construction.} We will perform the construction in two steps, by first giving a homotopy $K_1\colon F'_\ast\to F''_\ast$ to an auxiliary map $F''_\ast$ that is implicit in $F$ itself, and then a homotopy $K_2\colon F''_\ast\to (H^F_0)_\ast$ that will be induced by $H^F$. 

In order to construct $K_1$, let $(\gamma,1)\in\down{M}{N}_0\subset\EEx_1(\mfrak{S})$ be an exit $1$-path starting at $L$. We first observe that $(\gamma,1)$ is canonically a composition of $\boxed{\gamma_0}=(s_0\gamma(0),1)$ and $\gamma$, and this is witnessed by the exit $2$-path $(s_0\gamma,1)\in\EEx_2(\mfrak{S})$. Consequently, $F(\gamma,1)$ is canonically the composition of the image of a constant exit loop and the image of a path under $F|_N$:
\[
    (s_0\gamma,1)=
    \left(
    \begin{tikzcd}[]
        \bullet\ar[r,"{\gamma}"] & \bullet\\
        \bullet\ar[u,"{\boxed{\gamma_0}}"]\ar[ur,"{(\gamma,1)}"'] &
    \end{tikzcd}
    \right)
    \mapsto F(s_0\gamma,1)=
    \left(
    \begin{tikzcd}
        \bullet\ar[r,"{F|_N\gamma}"] & \bullet \\
        \bullet\ar[u,"{F_*\boxed{\gamma_0}}"]\ar[ur,"{F(\gamma,1)}"']
    \end{tikzcd}
    \right).
\]
Note that $F_*\boxed{\gamma_0}=h^F(\gamma(0))$ by definition. Now, the map
\begin{align*}
    F''_\ast\colon\down{L}{N}&\xrightarrow{(h^F\circ\evv_0)\times (F|_N)_*} \Map(\Delta^1,N')\times_{N'}\Map(\Delta^1,N')\\
    &\cong\Map(\Delta^1_1,N') \\
    &\xrightarrow{E_1^*}  \Map(\Delta^1,N')
\end{align*}
factors by construction through $\down{L'}{N'}\subset\Map(\Delta^1,N')$. Here, the left leg of the pullback is evaluation at $1$ and the right leg is evaluation at $0$, and the map $E_1\colon\Delta^1\to\Delta^1_1=\Delta^1\vee\Delta^1$ is the natural map defined in \Cref{O8ZQLHX}. We have thus defined the map $F''_\ast\colon\down{L}{N}\to\down{L'}{N'}$. The consideration above shows that it is homotopic to $F'_\ast$; an explicit homotopy $K_1\colon F'_*\to F''_*$ is given by 
\begin{align*}
    K_1\colon\down{L}{N}&\to \Map(\Delta^1,\down{L'}{N'})\\
    \gamma&\mapsto F'_*\mcal{C}_1^*s_0\gamma.
\end{align*}
Indeed, the map $\Phi(\gamma,1)\mapsto F'_*\Phi\mcal{C}^*_1(s_0\gamma,1)$ is the desired homotopy, but $\Phi(\gamma,1)=\gamma$ and it is a direct check that $\Phi\mcal{C}^*_1(s_0\gamma,1)=\mcal{C}^*_1s_0\gamma$ using the identities $\mcal{S}_1^*\mcal{C}^*_1=\id$ and $\mcal{S}_2^*\mcal{C}^*_1=s_0d_1$ to show that the two non-degenerate $2$-simplices of the square $\Phi\mcal{C}_1^*(s_0\gamma,1)$ are both given by $s_0\gamma$.\footnote{This shows that $\mcal{S}_2^*\mcal{C}_1^*(s_0\gamma,1)=(s_0\gamma,2)$, hence $\Phi$ applied to it yields $s_0\gamma$.}

Finally, $K_2\colon F_*''\to (H^{F}_0)_*$ is given by gluing the homotopy $h^F\circ\evv_0\to h^F_0$, given by $h^F=H^F|_L$ itself, with the homotopy $(F|_N)_*\to (H^F_0)_*$, likewise given by $H^F$ itself (or more precisely by $(H^F)^{-1}$). Concatenating $K_1$ and $K_2$ yields the desired homotopy $F'_*\to (H^F_0)'_*$. Since both $K_1$ and $K_2$ cover $(\id_{h^F_0},H^F)$ by construction, we obtain the full homotopy $\mbf{H}\colon F\to F^\tame_!$. 
\end{proof}

The following natural construction will let us generalise the method of proof of \Cref{4TH5BVK}.

\begin{construction}\label{AUCLOZE}
    Recall the bijections 
    \[
        \Hom_{\mrm{sSet}}(A\times B,C)\cong\Hom_{\mrm{sSet}}(A,\Fun(B,C))
    \]
    which induce an isomorphism $\Fun(A,\Fun(B,C))\cong\Fun(A\times B,C)$ of simplicial sets for any three simplicial sets $A$, $B$, $C$. Consequently, any map $F\colon\mscc{C}\times X\to\mscc{D}$ induces a map 
    \[
        F_*\colon\Ar(\mscc{C})\times X\to\Ar(\mscc{D})
    \]
    by writing $F$ as $\mscc{C}\to\Fun(X,\mscc{D})$, obtaining $\Ar(\mscc{C})\to\Ar(\Fun(X,\mscc{D}))\cong \Fun(\Delta[1]\times X,\mscc{D})\cong\Fun(X,\Ar(\mscc{D}))$ and finally rewriting as $\Ar(\mscc{C})\times X\to\Ar(\mscc{D})$.
\end{construction}

Before we proceed to the proof of \Cref{JG5DRK5}, we will prove a generalisation of \Cref{4TH5BVK} which is a corollary to its proof technique. Given a space $X$ and a linked space $\mfrak{S}$, let us denote by $\mfrak{S}\times X$ the linked space given by term-wise cartesian product with $X$.

\begin{lemma}\label{IX9X4IU}
    Given a space $X$ and a map $\gamma\colon\EEx(\mfrak{S})\times X\to \EEx(\mfrak{T})$, there exists a homotopy 
    \[
        \mbf{H}\colon \EEx(\mfrak{S})\times X\times \Delta^1\to\EEx(\mfrak{T})
    \]
    from $\gamma$ to $\gamma^\tame_!=\mbf{H}|_1$, where
    \[
        \gamma^\tame\colon \mfrak{S}\times X\to\mfrak{T}
    \]
    is a tame map. Moreover, for every point $x\in X$ at which the restriction $\gamma|_x\colon\EEx(\mfrak{T})\to\EEx(\mfrak{T})$ is tame, the restricted homotopy 
    \[
        (\mbf{H}|_x\colon\gamma|_x\to(\gamma^\tame|_x)_!)\colon\EEx(\mfrak{S})\times\Delta^1\to\EEx(\mfrak{T})
    \]
    is also tame.
\end{lemma}
\begin{proof}
    Suppose $\gamma\in\mscc{F}_{01}$ in the sense that the restriction $\gamma|_{Y\times X}$ factors through $Y'$ for $Y=M,N$. The remaining cases are analogous. Using \Cref{AUCLOZE} we obtain the map 
    \[
        \gamma_*\colon\Ar(\EEx(\mfrak{S}))\times X\to\Ar(\EEx(\mfrak{T}))
    \]  
    which induces the map 
    \[
        h^\gamma\colon L\times X\xhookrightarrow{\Box\times\id}\down{M}{N}\times X\xrightarrow{\gamma_*}\down{M'}{N'}\cong\down{L'}{N'}\subset (N')^{\Delta^1}
    \]
    which in turn yields the commuting square 
    \[
    \begin{tikzcd}
        L\times X\ar[d,hook,"{\iota\times\id}"']\ar[r,"{h^\gamma}"] & (N')^{\Delta^1}\ar[d,"{\evv_1}"]\\
        N\times X\ar[r,"{\gamma|_N}"']\ar[ur,dotted,"{H^\gamma}"] & N'
    \end{tikzcd}
    \]
    with a homotopy extension $H^\gamma$. We set 
    \[
        \gamma^\tame=(\gamma|_{M\times X},h^\gamma_0,H^\gamma_0)\colon \mfrak{S}\times X\to \mfrak{T}.
    \]
    The desired homotopy in the form 
    \[
        \mbf{H}\colon\EEx(\mfrak{S})\times X\to\Fun(\Delta[1],\EEx(\mfrak{T}))
    \]
    can now be constructed similarly to \Cref{4TH5BVK}. We will provide the details for completeness. 
    
    We claim that the derivatives 
    \(
        \gamma_*',(H^{\gamma}_0)_*\colon\down{L}{N}\times X\to\down{L'}{N'}
    \)
    are homotopic over $(\id_{h^\gamma_0},H^\gamma)\colon L \times N\times X\to (L'\times N')^{\Delta^1}$ via a homotopy 
    \[
        \mbf{H}\colon\down{L}{N}\times X\to{\down{L'}{N'}}^{\Delta^1}
    \]
    which will, for $((\gamma,e),\alpha)\in\EEx_{k}(\mfrak{S})\times X_k$, yield the desired homotopy
    \[
        \Psi'\mcal{S}_{e}^*\mbf{H}(\Phi\mcal{C}_e^*(\gamma,e),\alpha)\colon \mcal{S}_e^*\Psi'\gamma_*'(\Phi\mcal{C}_e^*(\gamma,e),\alpha)\to\mcal{S}_e^*\Psi'(H^\gamma_0)_*(\Phi\mcal{C}_e^*(\gamma,e),\alpha)
    \]
    with $\mcal{S}_e^*\mbf{H}$ defined as in \eqref{NCOTBV1}. Its first half $K_1\colon \gamma_*'\to\gamma_*''$ is constructed as follows. The maps 
    $\down{L}{N}\times X\xrightarrow{\evv_0\times\id}L\times X\xrightarrow{h^\gamma}\Map(\Delta^1,N')$ and $\down{L}{N}\times X\xrightarrow{(\gamma|_{N\times X})_*}\Map(\Delta^1,N')$ yield 
    \[
        \gamma_*''\colon \down{L}{N}\times X\xrightarrow{(h^\gamma\circ(\evv_0\times\id))\vee (\gamma|_{N\times X})_*}\Map(\Delta^1_1,N')\xrightarrow{E_1^*}\Map(\Delta^1,N').
    \]
    For each $\delta\in\down{L}{N}$ we have $\mcal{C}_1^* s_0\delta\in\Map(\Delta^1,\down{L}{N})$, applying $\gamma_*'\colon\down{L}{N}\to\Map(X,\down{L'}{N'})$ to which yields $K_1\colon \down{L}{N}\to\Map(\Delta^1,\Map(X,\down{L'}{N'}))$, $\delta\mapsto\gamma_*'\mcal{C}_1^*s_0\delta$. Finally, the second half $K_2\colon\gamma_*''\to(H^\gamma_0)_*$ is induced by $h^\gamma$ and $(H^\gamma)^{-1}$ exactly as before.

    As for the second statement, suppose 
    \[
        \gamma|_x=(\gamma|_{M\times\{x\}},(\gamma|_x)_L,\gamma|_{N\times\{x\}})\colon\mfrak{S}\to\mfrak{T}
    \]
    is a span map and observe, using the fact that $\Phi\circ \Box=\mrm{const}$ is the ordinary constant loop inclusion, that 
    \[
    \begin{tikzcd}
        L'\ar[r,"{\mrm{const}}"] & (L')^{\Delta^1}\ar[d,"{\iota'}"]\\
        L\times\{x\} \ar[u,"(\gamma|_x)_L"]\ar[r,"{h^\gamma|_x}"] & (N')^{\Delta^1} 
    \end{tikzcd}
    \]
    commutes, as does therefore 
    \[
    \begin{tikzcd}[column sep=large]
        L\times\{x\}\ar[d,"{\iota}"]\ar[r,"{\mrm{const}\circ (\gamma|_x)_L}"] & (L')^{\Delta^1}\ar[d,"{\iota'}"]\\
        N\times\{x\}\ar[r,"{H^\gamma|_x}"] & (N')^{\Delta^1}
    \end{tikzcd}
    \]
    using the defining property of $H^\gamma$. In particular we have $h^\gamma_0|_x=(\gamma|_x)_L$, implying 
    \[
        (\gamma^\tame|_x)=(\gamma|_{M\times\{x\}},(\gamma|_{x})_L,H^\gamma_0|_x).
    \]
    Consequently, since $H^{\gamma}_1|_x=\gamma|_{N\times\{x\}}$ by construction, we obtain the homotopy 
    \begin{equation}\label{YG031HV}
        \mcal{H}_x=(\id,\id,H^\gamma|_x)\colon\gamma^\tame|_x\to\gamma|_x
    \end{equation}
    of spans maps $\mfrak{S}\to\mfrak{T}$. We claim not only that $\mscc{H}|_x$ is tame, but that it is tamely homotopic to $(\mcal{H}_x)^{-1}$.

    Now, the map
    \[
        h^\gamma\circ(\evv_0\times\id)|_x\colon\down{L}{N}\times\{x\}\to\Map(\Delta^1,N')
    \]
    coincides with $\mrm{const}\circ (\gamma|_x)_L\circ\evv_0$ and so $\gamma_*''$ coincides with $(\gamma|_{N\times X})_*$ up to reparametrisation (due to the pre-composition with $E_1$), hence $K_1|_x$ is the constant homotopy up to reparametrisation. Similarly, $K_2|_x$ is induced only by $H^\gamma$ up to reparametrisation, hence so is $\mbf{H}|_x$. In particular, writing  $\mbf{H}^\tame|_x$ for the underlying tame homotopy of $\mbf{H}|_x=(\mbf{H}^\tame|_x)_!$, reparametrisation provides a homotopy
    \begin{equation}\label{N6688LO}
        (\mcal{H}_x)^{-1}\sim\mbf{H}^\tame|_x,
    \end{equation}
    a path in $\Hom_{\Hom(\mfrak{S},\mfrak{T})}(\gamma|_x,\gamma^\tame|_x)$.
\end{proof}

\begin{proof}[Proof of \Cref{JG5DRK5}]
    By \Cref{4TH5BVK}, it remains to show that the map $\Hom_{\mscc{LS}}(\mfrak{S},\mfrak{T})\to\Hom_{\Ctop}(\EEx(\mfrak{S}),\EEx(\mfrak{T}))$ itself is fully faithful, that is, that for all $F,G\colon\mfrak{S}\to\mfrak{T}$ the map
    \[
        \Hom_{\Hom_{\mscc{LS}}(\mfrak{S},\mfrak{T})}(F,G)\to\Hom_{\Hom_{\Ctop}(\EEx(\mfrak{S}),\EEx(\mfrak{T}))}(F_!,G_!)
    \]
    is an equivalence. We will first show that it is essentially surjective.

    Let $\gamma\colon \EEx(\mfrak{S})\times\Delta^{1}\to\EEx(\mfrak{T})$ be a path $F_!\to G_!$, an object of the target. By \Cref{IX9X4IU} we obtain a homotopy 
    \[
        (\mbf{H}\colon\gamma\to\gamma^\tame_!)\colon\EEx(\mfrak{S})\times\Delta^1\times\Delta^1\to\EEx(\mfrak{S}),
    \] 
    where $\gamma^\tame=(\gamma|_{M\times\Delta^1},h^\gamma_0,H^{\gamma}_0)\colon\gamma^\tame_0\to\gamma^\tame_1$ is a path in $\Hom_{\mscc{LS}}(\gamma^\tame_0,\gamma^\tame_1)$, and in fact $\gamma^\tame_0=(F_M,F_L,H^{\gamma}_0|_0)$ and $\gamma^\tame_1=(G(m),G_L,H^{\gamma}_0|_1)$. Moreover, $\mbf{H}$ restricts to the paths
    \[
        \mbf{H}_0=(\mbf{H}|_{\{0\}\times\Delta[1]}\colon F_!\to (\gamma^\tau_0)_!),\ \mbf{H}_1=(\mbf{H}|_{\{1\}\times\Delta[1]}\colon G_!\to(\gamma^\tau_1)_!)
    \] 
    in $\Hom_{\Ctop}(\EEx(\mfrak{S}),\EEx(\mfrak{T}))$ and provides equivalently a path 
    \begin{equation}\label{MH4KTHL}
        \mbf{H}\colon \gamma\to (\mbf{H}_0\ast\gamma^\tau_!\ast\mbf{H}^{-1}_1)
    \end{equation}
    in $\Hom_{\Hom_{\Ctop}(\EEx(\mfrak{S}),\EEx(\mfrak{T}))}(F_!,G_!)$.

    On the other hand, we have the homotopies $\mcal{H}_0\colon\gamma^\tame_0\to F$ and $\mcal{H}_1\colon\gamma^\tame_1\to G$ as in \eqref{N6688LO}, yielding the homotopy equivalence 
    \[
        \Upsilon\colon\Hom_{\Hom_{\mscc{LS}}(\mfrak{S},\mfrak{T})}(\gamma^\tame_0,\gamma^\tame_1)\xrightarrow{\sim} \Hom_{\Hom_{\mscc{LS}}(\mfrak{S},\mfrak{T})}(F,G),
    \]
    \(
        \delta\mapsto \mcal{H}_0^{-1}\ast\delta\ast\mcal{H}_1.
    \)
    This yields the square 
    \[
    \begin{tikzcd}
        \Hom_{\Hom_{\mscc{LS}}(\mfrak{S},\mfrak{T})}(F,G)\ar[r]\ar[d,"{\Upsilon^{-1}}"] & \Hom_{\Hom_{\Ctop}(\EEx(\mfrak{T}),\EEx(\mfrak{T}))}(F_!,G_!)\\
        \Hom_{\Hom_{\mscc{LS}}(\mfrak{S},\mfrak{T})}(\gamma^\tame_0,\gamma^\tame_1)\ar[r] & \Hom_{\Hom_{\Ctop}(\EEx(\mfrak{T}),\EEx(\mfrak{T}))}((\gamma^\tame_0)_!,(\gamma^\tame_1)_!) \ar[u,"{\Upsilon_!}"]
    \end{tikzcd}
    \]
    which homotopy-commutes due to the functoriality of $\EEx$. 
    
    Let now $\Gamma^\tame\colon F\to G$ be a path such that $\Upsilon(\gamma^\tame)\sim\Gamma^\tame$.
    We claim that $\Gamma^\tame_!\sim\gamma$, proving the essential surjectivity of the map $\Hom(F,G)\to\Hom(F_!,G_!)$. Indeed, since $(\mcal{H}_x^{-1})_!\simeq\mbf{H}_x$ for $x=0,1$ by \eqref{N6688LO}, we observe, using \eqref{MH4KTHL}, the homotopy 
    \(
        \Gamma^\tame_!\sim\Upsilon_!(\gamma^\tame_!)=(\mcal{H}_0^{-1})_!\ast\gamma^\tame_!\ast(\mcal{H}_1^{-1})_!\sim \gamma.
    \)

    What we have proved so far can be reformulated as follows. \Cref{4TH5BVK} shows that 
    \[
        \pi_0\EEx\colon\pi_0\Hom_{\mscc{LS}}(\mfrak{S},\mfrak{T})\to\pi_0\Hom_{\Ctop}(\EEx(\mfrak{S}),\EEx(\mfrak{T}))
    \]
    is surjective, and we have proved that for any $F,G\in\Hom(\mfrak{S},\mfrak{T})$, the map 
    \[
        \Pi_1\EEx\colon\pi_0\Hom(F,G)\to\pi_0\Hom(F_!,G_!)
    \] 
    is also surjective. In particular this implies that $\pi_0\EEx$ is injective and so bijective, and that $\pi_1\EEx$, with respect to any basepoint, is surjective. 
    The fact that $\EEx$ induces isomorphisms on all homotopy groups can now be seen in a similar fashion. 
    
    Let $n\geq1$ and assume by induction that $\pi_i\EEx$ is an isomorphism for $i\in\{0,\dots,n-1\}$ and that $\pi_{n}\EEx$ is surjective. It suffices to prove, for every $F,G\in\Omega^{n}\Hom(\mfrak{S},\mfrak{T})$ in the $n$-fold loop space of $\Hom(\mfrak{S},\mfrak{T})$ with respect to any basepoint, that the map 
    \[
        \Pi_{n+1}\EEx\colon\pi_0\Hom_{\Omega^{n}\Hom(\mfrak{S},\mfrak{T})}(F,G)\to\pi_0\Hom_{\Omega^{n}\Hom(\EEx(\mfrak{S}),\EEx(\mfrak{T}))}(F_!,G_!)
    \]
    is surjective, yielding in particular the injectivity of $\pi_{n}\EEx$. 
    
    Let us write $\mbb{S}^n$ for the simplicial $n$-sphere and let
    \(
        \gamma\colon\EEx(\mfrak{S})\times\mbb{S}^n\times\Delta^1\to\EEx(\mfrak{T})
    \) 
    be a path $F_!\to G_!$, representing an element in the target of $\Pi_{n+1}\EEx$. By \Cref{IX9X4IU} we obtain a homotopy
    \[
        (\mbf{H}\colon\gamma\to\gamma^\tame_!)\colon\EEx(\mfrak{S})\times\mbb{S}^n\times\Delta^1\times\Delta^1\to\EEx(\mfrak{T}),
    \]
    where the path 
    \(
        \gamma^\tame=(\gamma|_{M\times\mbb{S}^n\times\Delta^1},h^\gamma_0,H^\gamma_0)\colon\gamma^\tame_0\to\gamma^\tame_1
    \)
    is tame. Note that the second statement of \Cref{IX9X4IU} also holds for subspaces with the same proof: if $Y\subset X$ is a subspace such that the restriction $\gamma|_Y\colon\EEx(\mfrak{S})\times Y\to\EEx(\mfrak{T})$ is tame in the sense that $\gamma=\gamma^{\tame}_!$ for a tame map $\gamma^\tame\colon\mfrak{S}\times Y\to\mfrak{T}$, then the restricted homotopy $\mbf{H}_Y\colon\EEx(\mfrak{S})\times Y\times\Delta^1\to\EEx(\mfrak{T})$ is also tame. The construction of a homotopy-preimage $\Gamma^\tame$ in $\Hom_{\Omega^n\Hom(\mfrak{S},\mfrak{T})}(F,G)$ proceeds now as above upon replacing $x=0,1\in\Delta^1$ everywhere by the subspaces $Y=\mbb{S}^n\times\{i\}\subset X=\mbb{S}^n\times\Delta^1$, $i\in\{0,1\}$. This proves that $\Pi_{n+1}\EEx$ is surjective and that $\pi_n\EEx$ is injective.
\end{proof}

\printbibliography

\end{document}